\documentclass{article}
\usepackage{hyperref}
\usepackage{graphicx}
\usepackage[utf8]{inputenc}
\usepackage[T1,T2A]{fontenc}
\usepackage[russian,english]{babel}
\usepackage{xcolor}
\usepackage{gensymb}
\usepackage{amsmath,amsthm,amssymb,amsfonts}
\usepackage{enumerate}

\usepackage[colorinlistoftodos]{todonotes}

\usepackage{biblatex}
\newtheorem{theorem}{Theorem}

\newtheorem{lemma}{Lemma}
\newtheorem{prop}{Proposition}

\newcommand{\gf}{\mathfrak{g}}

\newcommand{\slnn}{\mathfrak{sl}_{n+1}}
\newcommand{\gln}{\mathfrak{gl}_{n}}
\newcommand{\sln}[1]{\mathfrak{sl}_{#1}}
\newcommand{\son}{\mathfrak{so}_{2n+1}}
\newcommand{\spn}{\mathfrak{sp}_{2n}}
\newcommand{\sond}{\mathfrak{so}_{2n}}

\newcommand{\dx}{\mathrm{dx}}
\newcommand{\ds}{\mathrm{ds}}
\newcommand{\dt}{\mathrm{dt}}
\newcommand{\dy}{\mathrm{dy}}

\DeclareMathOperator{\supp}{supp}

 \newenvironment{dedication}
   {
    \itshape             
    \raggedleft          
   }
   {\par 
   }

\begin{document}
\title {
  Limit shape for infinite rank limit of tensor power decomposition for Lie
  algebras of series $\son$.}
\author{A. Nazarov, P. Nikitin, O. Postnova }\date{}

\maketitle
\begin{dedication}
\;Dedicated to the memory of Vladimir Dmitrievich Lyakhovsky ${(1942-2020)}$
\end{dedication}

\begin{abstract}
  We consider the Plancherel measure on irreducible components of
  tensor powers of the spinor representation of $\son$. The
  irreducible representations correspond to the generalized Young
  diagrams. With respect to this measure the probability of an
  irreducible representation is the product of its multiplicity and
  dimension, divided by the total dimension of the tensor product. We
  study the limit shape of the generalized Young diagram when the
  tensor power $N$ and the rank $n$ of the algebra tend to infinity
  with $N/n$ fixed. We derive an explicit formula for the limit shape
  and prove convergence to it in probability. We prove central limit
  theorem for global fluctuations around the limit shape.
  
\end{abstract}
\section*{Introduction}

The emergence of the limit shapes of the random Young diagrams goes
back to Ulam's problem on the length of the maximal increasing
subsequence in a uniform random sequence \cite{ulam1961monte}. Through
the use of Robinson-Schensted-Knuth correspondence
\cite{robinson1938representations,schensted1961longest,knuth1970permutations}
a pair of Young tableaux is associated to the random sequence. Then
the length of the maximal increasing subsequence is equal to the
length of the first row of Young diagram. Uniform distribution on
number sequences after RSK mapping gives rise to Plancherel measure on
Young diagrams \cite{vershik1977asymptotics,logan1977variational}.


The limit shape problem for the Young diagram could also be stated for the
tensor product decomposition of irreducible representations of semisimple Lie
algebras. Due to the Schur-Weyl duality the multiplicities of the irreducible
components in the $N$-th tensor power of the vector fundamental representation
of $\slnn$ are the dimensions of the irreducible representations of $S_N$. The
Plancherel-type measure (\ref{eq:3}) associated with this decomposition was
first considered by Kerov \cite{kerov1986asymptotic}. The asymptotic behavior
of this measure was studied in three regimes: $N\to\infty$ with $n$ fixed,
$N\to\infty, n\to\infty$ with $N/n$ fixed and $N,n\to\infty$ with
$N/n^{2}$ fixed. The first case  was studied \cite{kerov1986asymptotic}
and later generalized to all simple Lie algebras in
\cite{nazarov2018limit,Postnova2019,tate2004lattice}. For the second case
Kerov discovered that Vershik-Kerov-Logan-Shepp limit shape of Young
diagrams with respect to the Plancherel measure on $S_{N}$ as $N\to\infty$
also appears as the limit shape with respect to this measure. Later, P. Biane
\cite{biane2001approximate,biane1998representations} described the limit
shapes for the third case. But the asymptotical behavior of the
Plancherel measure for $N$-th tensor powers of representations of Lie
algebras of types $\son,\spn,\sond$ has not been studied yet in the
limit $N,n\to\infty$.

In the present paper we consider the statistics of irreducible
components in the $N$-th tensor power of the spinor representation
$V^{\omega_{n}}$ of the algebra $\son$ in the limit $N,n\to\infty$. We
derive the limit shape in the limit when $N\to\infty$ and $N/n$ is
finite. It is convenient to use the description of irreducible
components in terms of generalized Young diagrams
\cite{kashiwara1994crystal,nakashima1993crystal}, that we call below
``diagrams''. Our main result is most conveniently stated in
coordinates that correspond to the diagrams (see
Fig.~\ref{fig:Bn-young-diagram-and-limit-shape-n20-N200}):

\begin{theorem}
  \label{diagram-convergence-theorem}
  As $n\to\infty,\; N\to\infty,\; c=\lim_{n,N\to\infty}\frac{N+2n-1}{n}=\mathrm{const}$,
  the upper boundary $f_{n}$ of a rotated and scaled generalized Young diagram for a highest weight in the decomposition
  of tensor power of the spinor representation
  $\left(V^{\omega_{n}}\right)^{\otimes N}$ of simple Lie algebra $\son$
  into irreducible representations converges in probability in the supremum
  norm $||\cdot||_{\infty}$ to the limiting shape given by the formula
 $f(x)=1+\int_{0}^{x}(1-4\rho(t))\dt$,
 where the limit density $\rho(x)$ is written explicitly as
\begin{equation}
  \label{eq:68}
  \rho(x)= \theta\left(\frac{c}{2}-|x|\right)      \frac{1}{2\pi}\Re\left[\arccos\frac{c-4}{\sqrt{c^{2}-4x^{2}}}\right]
\end{equation}
where $\theta\left(\frac{c}{2}-|x|\right)$ is the Heaviside step function.

That is, for all $\varepsilon>0$ we have
  \begin{equation}
   \label{eq:96}
    \mathbb{P}\left(\sup_{x\in\mathbb{R}}|f_{n}(x)-f(x)|>\varepsilon\right)\xrightarrow[n\to\infty]{} 0
  \end{equation}
\end{theorem}

We also demonstrate, that it is easy to prove the central limit
theorem for the global fluctuations of diagrams that is presented
below. The proof is based upon the use of the biorthogonal ensembles
techniques \cite{breuer2017central}.
\begin{theorem}[Central limit theorem]
  \label{central-limit-theorem}
  For a linear statistics $X_{f}=\sum_{i=1}^{n} f(x_{i}^{2})$, where $f\in
  C^{1}([0,2c-4])$ and $x_{i}=\frac{a_{i}}{2n}$ are the
  midpoints of the (scaled) intervals, where the upper boundary of a
  random Young diagram, rotated $45\degree$ decreases, we have
  \begin{equation}
    \label{eq:clt}
    X_{f}-\mathbb{E}X_{f}\to \mathcal{N}\left(0,\sum_{k\geq 1}k |\hat{f}_{k}|^{2}\right),
  \end{equation}
  in distribution as $n,N\to\infty$ with $c=\lim\frac{N+2n-1}{n}$, where
  the Fourier coefficients $\hat{f}_{k}$ are defined as
  \begin{equation}
    \label{eq:107}
    \hat{f}_{k}=\frac{1}{2\pi}\int_{0}^{2\pi}f\left((c-2)(\cos \theta+1)\right)e^{-ik\theta}\mathrm{d\theta}.
  \end{equation}
\end{theorem}

The paper is organized as follows. In Section \ref{sec:statements} we
describe tensor power decomposition for spinor representation of
$\son$, that is used to introduce the probability measure. We fix the
required notations, introduce the coordinates $x$ and describe the
generalized Young diagrams and their boundaries $f_{n}$ that are used
to state Theorem \ref{diagram-convergence-theorem}. We then explain
the multiplicity formula by skew Howe duality and discuss the
insertion algorithm, that can be used to sample random Young diagrams.

The proof of the Theorem \ref{diagram-convergence-theorem} is
contained in the Sections \ref{sec:variational-problem} and
\ref{sec:conv-prob-meas}. In the Section \ref{sec:variational-problem}
the variational problem for the limit shape is stated and solved and
function \eqref{eq:68} is obtained. In Section
\ref{sec:conv-prob-meas} we prove the convergence of diagrams in
probability, thus completing the proof of the Theorem
\ref{diagram-convergence-theorem}. The proof of the Theorem
\ref{central-limit-theorem} is presented in Section
\ref{sec:centr-limit-theor}.

In Conclusion we state open problems related to the presented results.

\section{Notations and probability measure}
\label{sec:statements}
First we recall the definition of the Plancherel measure for tensor
products. Let $\gf$ be a simple finite-dimensional Lie algebra of rank
$n$ and $V^{\nu}$ be its irreducible finite-dimensional highest-weight
representation. Denote simple roots of $\gf$ by
$\alpha_{1},\dots, \alpha_{n}$ and fundamental weights by
$\omega_{1},\dots \omega_{n}$, $(\alpha_{i},\omega_{j})=\delta_{ij}$.
The Weyl group is denoted by $W$, the main Weyl chamber by $C_{0}$,
the root system by $\Delta$ and the set of positive roots by
$\Delta^{+}$. The root lattice $\mathbb{Z}\Delta$ is denoted by $Q$
and for a weight $\nu$ we denote by $Q(\nu)$ the set of weights that
are obtained by repeated subtractions of positive roots from $\nu$:
$Q(\nu)=\nu-\mathbb{Z}_{+}\Delta^{+}$.

Tensor power of $V^{\nu}$ is a completely reducible representation and can be decomposed as:
\begin{equation}
  \left(V^{\nu}\right)^{\otimes N}\cong \bigoplus_{\lambda\in Q(\nu)\cap C_{0}}
  W_{\lambda}(V^\nu, N)\otimes
  V^{\lambda},
\end{equation}

The sum is taken over irreducible components of the tensor product and
$W_{\lambda}(V^{\omega},N)$ is the space of multiplicities:
\begin{equation}
W_{\lambda}(V^\nu, N)\simeq \mathrm{Hom}_{\gf}( (V^{\nu})^{\otimes N} , V^\lambda).
    \end{equation}
Its dimension $M_{\lambda}^N\equiv M_{\lambda}^N(V^\nu)$ is the multiplicity of
$V^{\lambda}$ in the tensor product. \\
This decomposition gives the identity
\begin{equation}
  \label{eq:1}
  \left(\dim V^{\nu}\right)^{N}= \sum_{\lambda\in Q(\nu)\cap C_{0}} M^N_{\lambda} \dim V^{\lambda},
\end{equation}

This formula can be used to introduce the probability measure on the
set of dominant integral weights $Q(\nu)\cap C_{0}$. By the analogy
with the representation theory of the permutation group we call it
Plancherel measure:
\begin{equation}
  \label{eq:3}
  p^N(\lambda)=\frac{M^{N}_{\lambda}\dim{V^{\lambda}}}{\left(\dim{V^{\nu}}\right)^{N}}.
\end{equation}

From now on we focus on $\gf=\son$. In the standard orthogonal basis the simple roots are
$\{\alpha_i=e_i-e_{i+1}|i=1\dots n-1\}\cup \{ \alpha_n=e_n\}$. The
root system $B_n$ consists of the roots
$\Delta=\{\pm e_i\pm e_j|i\neq j\}\cup \{\pm e_i\}$, positive roots
are $\Delta^+=\{e_i+e_j|i<j\}\cup \{e_i\}\cup \{e_j-e_i|j<i\}$. The
fundamental weights of $B_n$ in the same basis are  given by formulae
\begin{equation}
  \begin{array}{l}
  \omega_1=e_1\\
  \omega_2=e_1+e_2\\
  \dots\\
  \omega_{n-1}=e_1+\dots + e_{n-1}\\
    \omega_n=\frac{1}{2}(e_1+\dots+e_n)
  \end{array}    
\end{equation}

We consider tensor powers $\left(V^{\omega_{n}}\right)^{\otimes N}$ of the last fundamental representation
$\nu=\omega_{n}$, that is also known as spinor representation.

Dominant integral weights $\lambda$ are linear combinations of
fundamental weights with non-negative integer coefficients $l_{n}$,
that are called Dynkin labels:
\begin{equation}
  \label{eq:2}
  \lambda=\sum_{i=1}^{n}l_{n}\omega_{n}.
\end{equation}
In orthogonal coordinates such a weight is written as
\begin{equation}
  \label{eq:57}
  \lambda=\sum_{i=1}^{n}\left(l_{i}+l_{i+1}+\dots+\frac{l_{n}}{2}\right)e_{i}.
\end{equation}

Dominant integral weights can be depicted by the generalized Young diagrams. For
algebras of series $\son$ it is convenient to use the diagrams with
boxes of two different widths, one being twice the other
\cite{nakashima1993crystal,kashiwara1994crystal} (see also
\cite{hong2002introduction}). Below we will the generalized Young
diagrams for the series $\son$ ``the diagrams''. In the present case the analogue of
Littlewood-Richardson rule for tensor product decomposition is more
difficult than for ordinary Young diagrams for $\sln{n}$ and number of
boxes in the diagram is not equal to the tensor power $N$. Since there
are boxes of two different widths it is  important to distinguish
between the number of boxes in a row and the length of the row. 
The length of the diagram's row $\lambda_{i}$ is equal to the corresponding orthogonal
coordinate. The number of boxes is equal to $\sum_{j=i}^{n}l_{i}$. In
such diagrams first $l_{n}$ boxes are of width one half. See an example in Fig~\ref{fig:young-bn}.
\begin{figure}[h]
  \includegraphics[width=10cm]{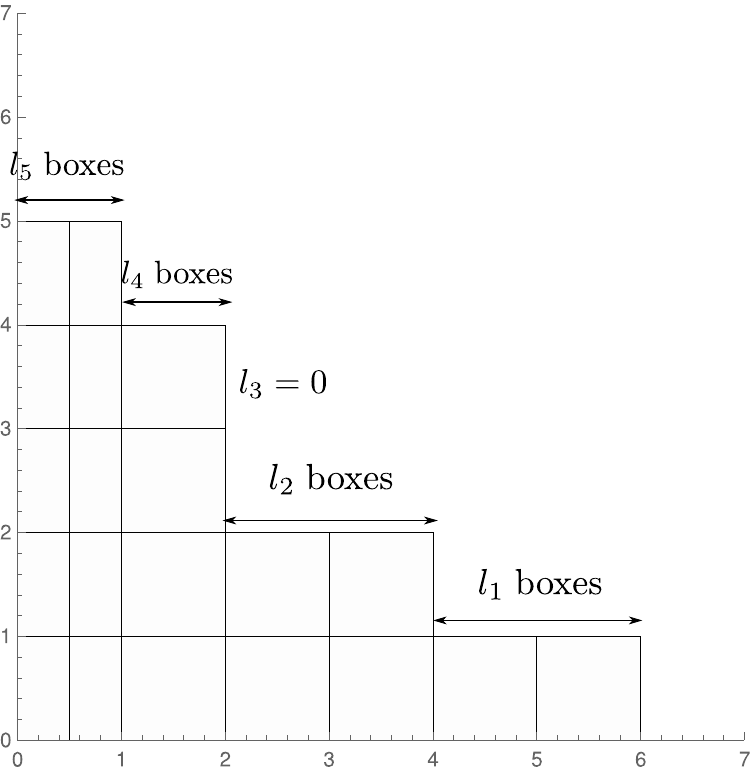}
  \caption{The generalized Young diagram for $B_5$ weight $\lambda$ with coordinates
    $[6,4,2,2,1]$ in orthogonal basis and Dynkin labels $(2,2,0,1,2)$}
  \label{fig:young-bn}
\end{figure}

For convenience we use the coordinates $\{a_{i}\}$ given by the formula
\begin{equation}
  \label{eq:6}
  a_{i}=2\sum_{j=i}^{n-1}l_{j}+l_{n}+2(n-i)+1.
\end{equation}
Such coordinates are positive integer numbers for integral dominant
weights and $a_{i}\geq a_{j}$ for $i<j$. 
The coordinates $\{a_{i}\}_{i=1}^{n}$ has natural interpretation if we
scale the diagram by the factor $2\sqrt{2}$, rotate it
45\degree\;  counterclockwise and shift it in such a way that lowest
point has coordinate $(2n)$. Then the upper border of the diagram is a
graph of piecewise linear function and $a_{i}$ is an $x$-coordinate of
the middle of decreasing piece number $i$, if we count decreasing
pieces from the right. See
Fig~\ref{fig:young-rotated-ai-geom-meaning}. 

\begin{figure}[h]
  \centering
  \includegraphics[width=10cm]{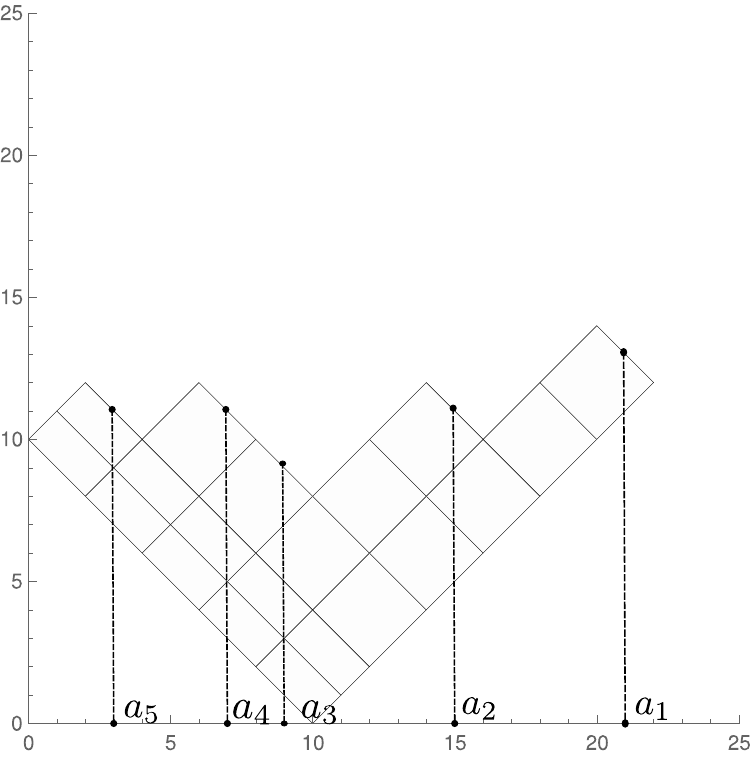}
  \caption{Rotated and scaled generalized Young diagram and the geometrical meaning of
    the coordinates $\{a_i\}_{i=1}^{n}$. }
  \label{fig:young-rotated-ai-geom-meaning}
\end{figure}

The probability measure on the integral dominant weights is introduced by the
formula (\ref{eq:3}) as
\begin{equation}
  \mu_{n,N}(\lambda)=\frac{M^N_{\lambda}\dim V^{\lambda}}{2^{nN}}.
  \label{eq:5}
\end{equation}

We consider the limit $N,n\to\infty$ such that ratio of $n$ and $N$ tends to a finite constant:
\begin{equation}
  \label{eq:15}
 c=\lim_{N,n\to\infty}\frac{N+2n-1}{n}, \quad c=\mathrm{const}.
\end{equation}
We are interested in the limiting probability distribution on
the irreducible components of the tensor power decomposition. Since
dominant integral weights are depicted by the diagrams, the measure
\eqref{eq:5} can be seen as a probability measure on the diagrams.
Therefore we are interested in the limit shape of generalized Young diagrams
with respect to the measure $\mu_{n,N}$.

Random diagrams with respect to the measure \eqref{eq:5} can be
efficiently sampled with the following algorithm, introduced by
Benkart and Stroomer \cite{benkart1991tableaux}. The algorithm uses
Sundaram tableaux for $\son$ and modified Berele insertion
\cite{berele1986schensted}, which works as follows. The tableau that
corresponds to some basis element in the irreducible representation
with the diagram $\lambda$ has the shape $\lambda$ that is filled with
the numbers ${1,\bar 1,2,\bar 2,\dots,n,\bar n,\infty}$ with the order
$1<\bar 1<2\bar 2\dots<\infty$, and can have first column of boxes
with width $\frac{1}{2}$ that can not contain $\infty$. The numbers
strictly increase along the columns and weakly increase along the rows
of the tableau, moreover, there are no numbers smaller than $\bar i$
below the $i$-th row. The tableau for the spinor representation
$V^{\omega_{n}}$ is the full column with the width $\frac{1}{2}$. Due
to the condition we have two choices $i,\bar i$ in the row number $i$,
so that total number of such tableaux is $2^{n}$ and they are in the
bijection with the basis elements of the spinor representation.

Tensor product decomposition
$V^{\lambda}\otimes V^{\omega_{n}}=\bigoplus_{\mu}V^{\mu}$ is
represented by the insertion of the column, corresponding to
$V^{\omega_{n}}$. Insertion of this column to the tableau that does
not contain first half-width column is done by erasing all the boxes
with $\infty$ and adjoining the column from the left. Insertion to the
diagram with the first half-width column is done by taking the
inserted column and the first column, reading them from the bottom up
row-by-row and producing a sequence of numbers from the two
half-columns by the following rule:
$(k,k)\to k, (\bar k,\bar k)\to \bar k, (k,\bar k)\to \emptyset, (\bar
k,k)\to \bar k, k$. We then insert these numbers one-by-one with a
modified Berele insertion, that inserts the number by the Schensted
insertion \cite{schensted1961longest,knuth1970permutations}, but if
the insertion of some number $i$ into some box should push the number
$\bar i$ from this box to some row below the row number $i$, then the
box is erased. The empty box is then shifted by jeu de taquin slides
to the edge of the diagram and is filled by $\infty$.

Therefore to sample the diagrams with the measure \eqref{eq:5} we
produce $N$ uniform random vectors of zeros and ones with length $n$,
that represent signs of the numbers in the half-width one-column
tableaux, and then insert these half-width one-column tableaux
one-by-one. In
Fig~\ref{fig:Bn-random-young-diagram-and-limit-shape-n50-N300} we
present a random diagram for $\mathfrak{so}_{101}$ and tensor power
$300$ sampled using this algorithm.

\begin{figure}[h]
  \centering
  \includegraphics[width=13cm]{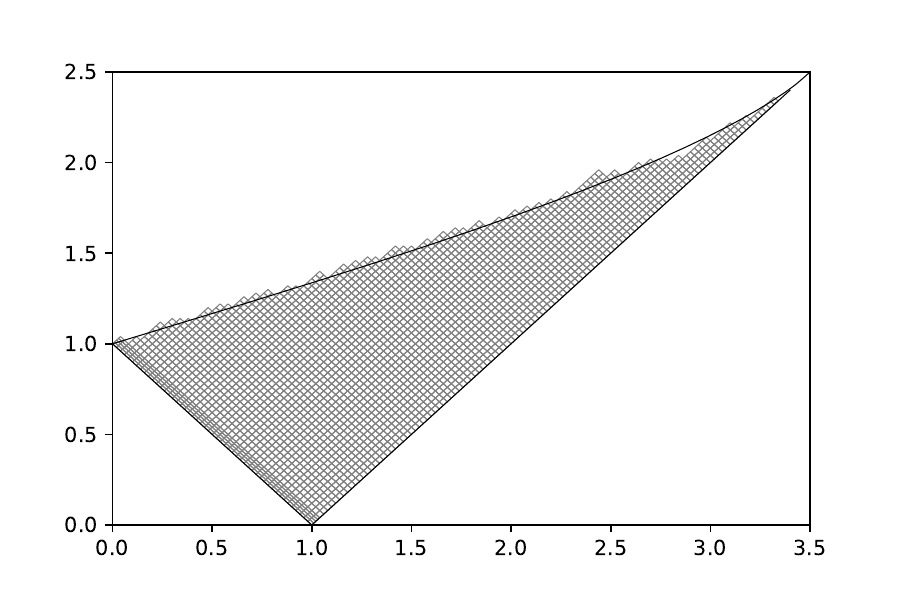}
  \caption{Rotated and scaled random diagram for $\mathfrak{so}_{101}$ and $N=300$
    and limit shape $f(x)$ of the generalized Young diagrams for $c=\frac{N}{n}+2=8$.}
  \label{fig:Bn-random-young-diagram-and-limit-shape-n50-N300}
\end{figure}

We scale the diagram by the
factor $\frac{\sqrt{2}}{n}$, rotate it $45\degree$ counterclockwise
and shift along $x$ axis in such a way that the lowest point has
coordinates $(1,0)$. This corresponds to a rescaling of the
coordinates $\{a_{i}\}$ as $x_{i}=\frac{a_{i}}{2n}$. See
Fig~\ref{fig:Bn-young-diagram-and-limit-shape-n20-N200} for an example
of the most probable diagram for $n=20$, $N=200$ and limit shape for
$c=12$. The upper border of the diagram is a graph of piecewise linear
function $f_{n}(x)$, which is almost everywhere differentiable and
$f_{n}'(x)=\pm 1$ if $x\neq \frac{i}{2n}$. We will prove that
piece-wise linear functions $f_{n}(x)$ converge in probability w.r.t.
to the probability measure \eqref{eq:5} to a continuous smooth
function $f(x)$ when $n\to\infty$.

\begin{figure}[h]
  \centering
  \includegraphics[width=10cm]{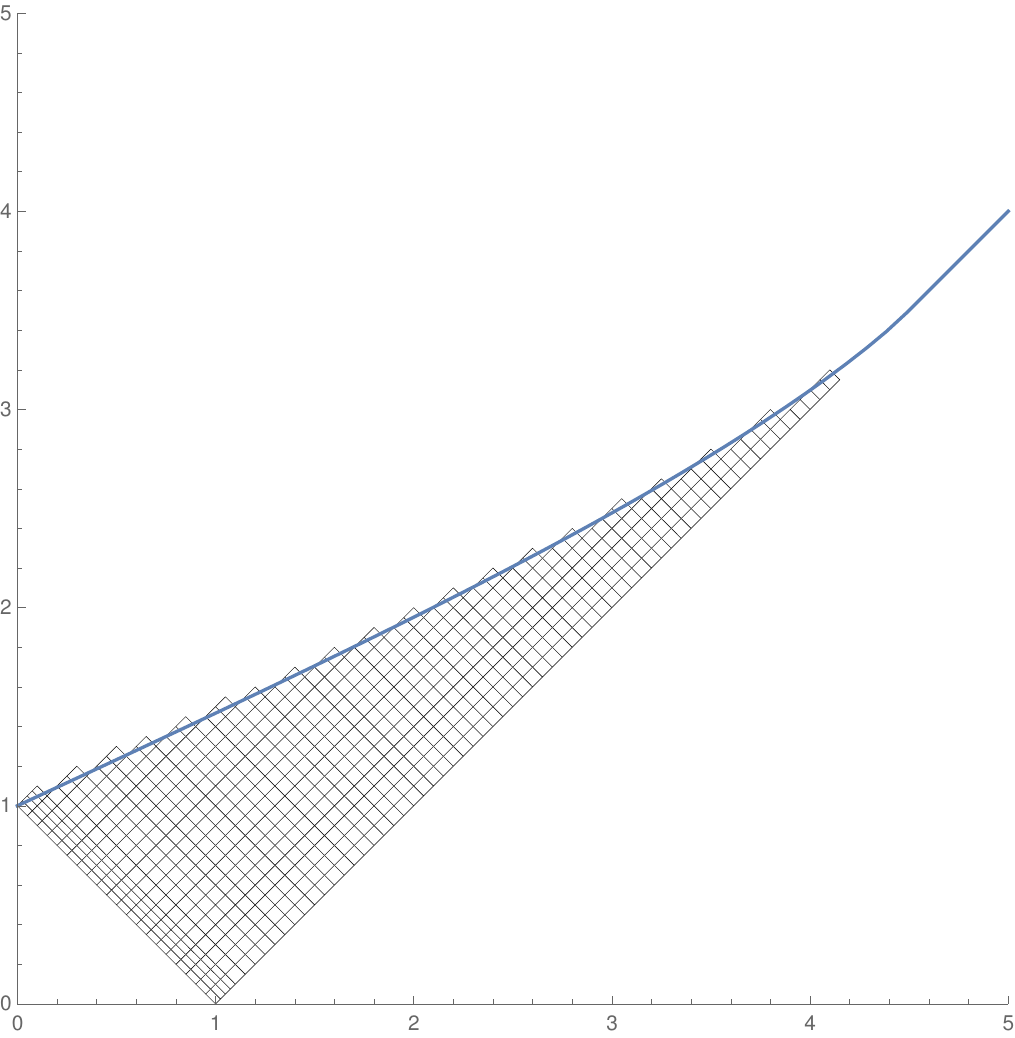}
  \caption{Rotated and scaled diagram for $B_{20}$ and $N=200$
    and limit shape $f(x)$ of the generalized Young diagrams for $c=\frac{N}{n}+2=12$.}
  \label{fig:Bn-young-diagram-and-limit-shape-n20-N200}
\end{figure}

To derive the limit shape it is convenient to consider the diagrams as
the particle point processes with coordinates $\{x_{i}\}_{i=1}^{n}$.
Introduce the piecewise constant function
$\rho_{n}(x)=\frac{1}{4}(1-f_{n}'(x))$. It is equal to zero on an
interval of the length $\frac{1}{n}$ if there is no particle in the
middle of the interval and is equal to $\frac{1}{2}$ if there is a
particle, which means that there is one particle on two intervals of
length $\frac{1}{2n}$. So the function $\rho_{n}(x)$ can be called
particle density.
The convergence of the diagrams to the limit shape leads to the
convergence of particle density functions $\rho_{n}$ to a limit
particle density $\rho(x)$.

Due to our choice of normalization, limit density $\rho(x)$ is
connected to a derivative of limit function $f(x)$ of the diagrams
by the formula
\begin{equation}
  \label{eq:61}
  f'(x)=1-4\rho(x),
\end{equation}
and limit shape can be recovered from the explicit expression for
$\rho(x)$ by the formula
\begin{equation}
  \label{eq:62}
  f(x)=1+\int_{0}^{x}(1-4\rho(t))\dt.
\end{equation}

It is more convenient to solve the variational problem for the limit
density $\rho(x)$.

In order to state the variational problem we need to write the
explicit formula for the probability measure $\mu_{n,N}$. To do so we
recall that P.P.~Kulish, V.D.~Lyakhovsky and O.V.~Postnova derived the
explicit formulae for tensor product decomposition multiplicities
using Weyl group symmetry and recurrence relations
\cite{kulish2012tensor,kulish2012tensorb2,kulish2012multiplicity,kulish2012multiplicityJPC}.


For the case of $\son$ and $\left(V^{\omega_{n}}\right)^{\otimes N}$,
for $\lambda $ written in the coordinates $\{a_{i}\}$ that are
described above, the formula is: 
 \begin{equation}
    \tilde{M}^{\omega_{n},N}_{\lambda(a_1\dots a_{n})}=
   \prod_{k=0}^{n-1}\frac{\left(
N+2k\right) !}{2^{2k}\left( \frac{N+a_{k+1}+2n-1}{2}\right) !\left( \frac{%
N-a_{k+1}+2n-1}{2}\right) !}\prod_{l=1}^{n}a_{l}\prod_{ i<j %
}\left( a_{i}^{2}-a_{j}^{2}\right).\rule{0mm}{8mm}
   \label{bn}
\end{equation}
 
Note that the factors in the numerator vanish at the boundaries of the
Weyl chambers, shifted by Weyl vector
$-\rho=-\sum_{i=1}^{n}\omega_{i}$ and the denominator provides that
$\tilde{M}^{\omega_{1},N}_{\lambda}$ satisfies the boundary conditions
and also ensures that the whole expression is anti invariant w.r.t.
Weyl group transformations.
   
Note that there are two congruence classes of weights, one is parametrized by even values of
$a_{i}$ while another by odd. The class is determined by the parity of $N$. For $N$ even we get
$a_{i}$ odd and vice versa.

This result can be understood as a consequence of the skew Howe
duality between $\son$ and $\mathfrak{o}_{N}$. By theorem, presented,
for example, in ~\cite{howe1989remarks}, the exterior algebra
$\bigwedge\left(\mathbb{C}^{2n+1}\otimes \mathbb{C}^{k}\right)$ admits
a multiplicity-free decomposition
$\bigwedge\left(\mathbb{C}^{2n+1}\otimes \mathbb{C}^{k}\right)\simeq
\bigoplus_{\lambda}V_{\son}^{\lambda}\otimes
V_{\mathrm{Pin}_{2k}}^{\bar\lambda'}$, where $\bar\lambda'$ is a
generalized Young diagram that is complement conjugate to the diagram
$\lambda$ inside of the $n\times k$-rectangle. From this decomposition
the multiplicity formula
$\left(V_{\son}^{\omega_{n}}\right)^{\otimes
  2k}\simeq\bigoplus_{\lambda}2^{1-k}\mathrm{dim}V_{\mathfrak{so}_{2k}}^{\bar\lambda'}
V_{\son}^{\lambda}$ can be deduced \cite{2021arXiv211112426N}. Setting
$N=2k$ and expressing the dimension of the
$\mathfrak{so}_{N}$-representation in terms of the coordinates
$\{a_{i}\}$ we arrive at the formula \eqref{bn}.

We use the Weyl dimension formula
\begin{equation}
\label{weyl}
\dim V^{\lambda}=\prod_{\alpha\in\Delta^{+}} \frac{(\lambda+\rho,\alpha)}{(\rho,\alpha)}=\frac{2^{-n^2+2n}n!}{(2n)!(2n-2)!\dots 2!}\cdot \prod_{i< j}(a_i^2-a_j^2)\prod_{l=1}^{n}a_l,
\end{equation}
thus we obtain the discrete probability measure  with the explicit density function (or the probability mass function):
\begin{multline}
\label{meas}
\mu_{n,N}(\lambda)=
\mu_{n,N}(\{a_i\})
=\frac{\tilde{M}^{\omega_{n},N}_{\lambda(a_1\dots a_{n})}\dim V^{\lambda}}{(2^n)^N}=\\
=
\prod_{k=0}^{n-1}\frac{\left(
N+2k\right) !}{2^{2k}\left( \frac{N+a_{k+1}+2n-1}{2}\right) !\left( \frac{%
N-a_{k+1}+2n-1}{2}\right) !}\prod_{i< j}(a_i^2-a_j^2)^2\prod_{l=1}^{n}a_l^2\cdot\frac{2^{-n^2+2n-nN}n!}{{(2n)!(2n-2)!\dots 2!}}.
\end{multline}  

\section{Variational problem for the limit shape}
\label{sec:variational-problem}
To prove the convergence of generalized Young diagrams to the limit shape, we need
to consider the upper boundaries of the rotated diagrams as functions $f_{n}$
with bounded derivative. Then we can prove convergence in a space of
such functions with respect to certain distance. Our approach to the proof of
convergence is very similar to the proof of Vershik-Kerov-Logan-Shepp
theorem in the book by Dan Romik \cite{romik2015surprising}.

We first rewrite the probability of configuration \eqref{eq:5} as an exponent
of a quadratic functional on rotated diagram boundaries. Looking for a minimum
of this functional, we obtain a variational problem that we solve in Lemmas
\ref{lemma:rho-for-c-greater-4},~\ref{lemma:rho-for-c-less-than-4},
\ref{lemma:explicit-limit-for-rho}.

\subsection{Probability of configuration as an exponent of a quadratic
  functional}
\label{sec:prob-conf-as}

We need to rewrite the formula for the probability measure in the form that is
more convenient for analysis. We do so in the following Lemma. 
  \begin{lemma}
    \label{lemma:conv-empir-meas} 

    Denote by $a_{-i},i>0$ a ``mirror image'' of $a_{i}$:
    \begin{equation}
      \label{eq:12}
      a_{-i}\equiv -a_{i},
    \end{equation}
    Then we can rewrite the measure  \eqref{meas}  in the form:
    \begin{equation}
      \label{eq:13}
      \mu_{n}(\left\{a_{i}\right\}_{i=-n,i\neq 0}^{n})=\frac{1}{Z_{n}}\prod_{i<j;i,j\neq 0;i,j=-n}^{n}|a_{i}-a_{j}|
      \cdot \prod_{l=-n,l\neq 0}^{n}\exp\left[-(2n)V_{0}\left(\frac{a_{l}}{2n}\right)-e_{n}\left(a_{l}\right)\right],
    \end{equation}
    where
    \begin{equation}
      V_{0}(u)=\frac{1}{4}\left[\left(\frac{c}{2}+u\right)\ln\left(\frac{c}{2}+u\right)+\left(\frac{c}{2}-u\right)\ln\left(\frac{c}{2}-u\right)\right],
    \end{equation}
    $e_{n}(u)=\frac{1}{4}\ln\left((cn)^{2}-u^{2}\right)+\frac{1}{2}\ln |u|+\mathcal{O}\left(\frac{1}{n}\right)$ and $Z_{n}$ does not
    depend on $a_{l}$ and  an additional condition \eqref{eq:12} is satisfied.

    The function  $V_{0}(u)$ is twice continuously differentiable on 
    $(-\frac{c}{2},\frac{c}{2})$, $|V_{0}''(u)|\leq C(1+|u-c/2|^{-1}+|u-c/2|^{-1})$ and $e_{n}$
    is uniformly bounded by $C\ln 2n$. 
    \begin{proof}
      We first extract the contribution that does not depend on $a_{k}$ from  the expression \eqref{meas} 
      \begin{equation}
        \label{eq:10}
        C_{n}=\frac{2^{-n^{2}+2n-nN}n!}{(2n)!(2n-2)!\dots 2!}\prod_{k=0}^{n-1} \frac{(N+2k)!}{2^{2k}}.
      \end{equation}
      The exponential term can be written for $l=1,\dots,n$ as
      \begin{equation}
        \label{eq:11}
        \exp\left[\ln\left(\left(\frac{N+a_{l}+2n-1}{2}\right)!\right)\right]\cdot\exp\left[\ln\left(\left(\frac{N-a_{l}+2n-1}{2}\right)!\right)\right]\cdot\exp\left[\ln|a_{l}|\right],
      \end{equation}
      note that we keep another product $\prod_{k} |a_{k}|$ outside of the exponent for future use. 
      Use the notation \eqref{eq:5}. First we use Stirling approximation formula for
      factorials to write the exponential term as
      \begin{multline}
        \label{eq:14}
        2\pi\exp\left[\frac{cn+a_{l}}{2}\ln\frac{cn+a_{l}}{2}-\frac{cn+a_{l}}{2}+\frac{1}{2}\ln\frac{cn+a_{l}}{2}+\right.\\
        \left.+\frac{cn-a_{l}}{2}\ln\frac{cn-a_{l}}{2}-\frac{cn-a_{l}}{2}+\frac{1}{2}\ln\frac{cn-a_{l}}{2}+\ln|a_{l}|+O\left(\frac{1}{n}\right)\right].
      \end{multline}
      Collecting the terms and combining in the leading contributions $a_{l}$ with $2n$, we get
      \begin{multline}
        \label{eq:16}
        2\pi\exp\left[cn\ln n +
          (2n)\cdot\frac{1}{2}\left(\left(\frac{c}{2}+\frac{a_{l}}{2n}\right)\ln\left(\frac{c}{2}+\frac{a_{l}}{2n}\right)+\left(\frac{c}{2}-\frac{a_{l}}{2n}\right)\ln\left(\frac{c}{2}-\frac{a_{l}}{2n}\right)\right)-cn+\right.\\
        +\left.\frac{1}{2}\ln \left((cn)^{2}-a_{l}^{2}\right) -\ln 2 +\ln|a_{l}|+O\left(\frac{1}{n}\right)\right].
      \end{multline}
      Now we denote by $V_{0}(u)$ (one half of) the main contribution with $u=\frac{a_{l}}{2n}$:
      \begin{equation}
        \label{eq:17}
        V_{0}(u)=\frac{1}{4}\left[\left(\frac{c}{2}+u\right)\ln\left(\frac{c}{2}+u\right)+\left(\frac{c}{2}-u\right)\ln\left(\frac{c}{2}-u\right)\right],
      \end{equation}
      and by $e_{n}$ (one half of) the remainder
      \begin{equation}
        \label{eq:18}
        e_{n}(u)=\frac{1}{4}\ln\left((cn)^{2}-u^{2}\right)+\frac{1}{2}\ln |u|+\mathcal{O}\left(\frac{1}{n}\right).
      \end{equation}
      We also get $a_{l}$-independent contribution for each $l=1,\dots,n$
      \begin{equation}
        \label{eq:19}
        \tilde{C}_{n}=\pi\exp\left[cn\ln n-cn+\mathcal{O}\left(\frac{1}{n}\right)\right].
      \end{equation}

      Combining the exponential terms \eqref{eq:16} with the notations
      \eqref{eq:17},\eqref{eq:18} and $a_{l}$-independent contributions
      \eqref{eq:19}, \eqref{eq:10}, we arrive at the expression of the form 
          \begin{multline}
      \label{eq:4}
      \mu_{n}(\left\{a_{i}\right\})=\frac{1}{\tilde{Z}_{n}}\prod_{i<j;i,j=1}^{n}(a_{i}-a_{j})^{2}\cdot\prod_{i<j;i,j=1}^{n}(a_{i}+a_{j})^{2}\cdot \\\cdot\prod_{k}|a_{k}|
      \cdot \prod_{l=1}^{n}\exp\left\{ 2\left[-(2n)\cdot V_{0}\left(\frac{a_{l}}{2n}\right)-e_{n}\left(a_{l}\right)\right]\right\},
    \end{multline}
    where $\frac{1}{\tilde{Z}_{n}}=\frac{C_{n}}{(\tilde{C}_{n})^{n}}$.

      Let us now introduce the mirror images  $a_{-i}\equiv -a_{i}$ and use
      the equality
      \begin{equation}
        \label{eq:26}
        \prod_{i<j;i,j\neq 0; i,j=-n}^{n} |a_{i}-a_{j}|=\prod_{i<j;i,j=1}^{n}|a_{i}-a_{j}|\cdot\prod_{i=-n}^{-1}\prod_{j=1}^{n}|a_{i}-a_{j}|\cdot
        \prod_{i<j;i,j=-n}^{-1}|a_{i}-a_{j}|.
      \end{equation}
      The second factor can be expanded as 
      \begin{equation}
        \label{eq:27}
        \prod_{i=-n}^{-1}\prod_{j=1}^{n}|a_{i}-a_{j}|=2^{n}\prod_{j=1}^{n}|a_{j}|\cdot \prod_{i\neq j;i,j=1}^{n}|a_{i}+a_{j}|=2^{n}\prod_{j=1}^{n}|a_{j}|\cdot \prod_{i< j}|a_{i}+a_{j}|^{2},
      \end{equation}
      while the first and the third terms coincide. Thus we see that
      \begin{equation}
        \label{eq:28}
        \prod_{i<j;i,j\neq 0; i,j=-n}^{n} |a_{i}-a_{j}|=2^{n}\prod_{i<j;i,j=1}^{n}\left(a_{i}^{2}-a_{j}^{2}\right)^{2}\cdot \prod_{k=1}^{n}|a_{k}|.
      \end{equation}

      So we can write the
      measure for $l=-n,\dots,-1,1,\dots,n$ in the form \eqref{eq:13} with
      \begin{equation}
        \label{eq:20}
        \frac{1}{Z_{n}}=C_{n}(\tilde{C}_{n})^{-n}2^{-n}.
      \end{equation}
      Note that for $V_{0}''(u)\leq A\left(1+\frac{1}{|u-c/2|}+\frac{1}{|u+c/2|}\right)$ for
      $u\in[-c/2,c/2]$ and $e_{n}$ is uniformly bounded by $B\ln n$ for some constants
      $A,B$, since $|a_{l}|\leq cn$ for $l=-n,\dots,-1,1,\dots n$.
    \end{proof}
\end{lemma}

Using Lemma \ref{lemma:conv-empir-meas} we can rewrite the probability
of a highest weight $\lambda$ in the limit $N,n\to\infty, N\sim n$ and
of the corresponding diagram as the exponent of the functional of the
rotated diagram's boundary $f_{n}(x)$:
\begin{equation}
  \label{eq:63}
  \mu_{n}(\lambda)=\mu_{n,N}\left(\{a_{i}\}_{i=1}^{n}\right)=e^{-(2n)^{2}J[f_{n}]+\mathcal{O}(n\ln n)}.
\end{equation}
In order to write down the functional $J[f]$ explicitly, we need to
recall that the density $\rho_{n}(x)$ is closely related to the derivative
$f_{n}'(x)$. We can interpret $\rho_{n}$ as a density of middle points
of small intervals of length $\frac{1}{n}$ (or $\frac{1}{2n}$ for
$i=n$) where $f_{n}'(x)=-1.$ Note also, that in order to obtain the
expression \eqref{eq:13} we had to continue the function $\rho_{n}(x)$
to negative values of $x$ so that it becomes an even function. This
corresponds to the continuation of the boundary $f_{n}$ such that
$f_{n}'$ is even and $f_{n}$ is continuous at $x=0$. The continuation
is shown in the figure \ref{fig:young-rotated-and-continued}. 

\begin{figure}[h]
  \centering
  \includegraphics[width=10cm]{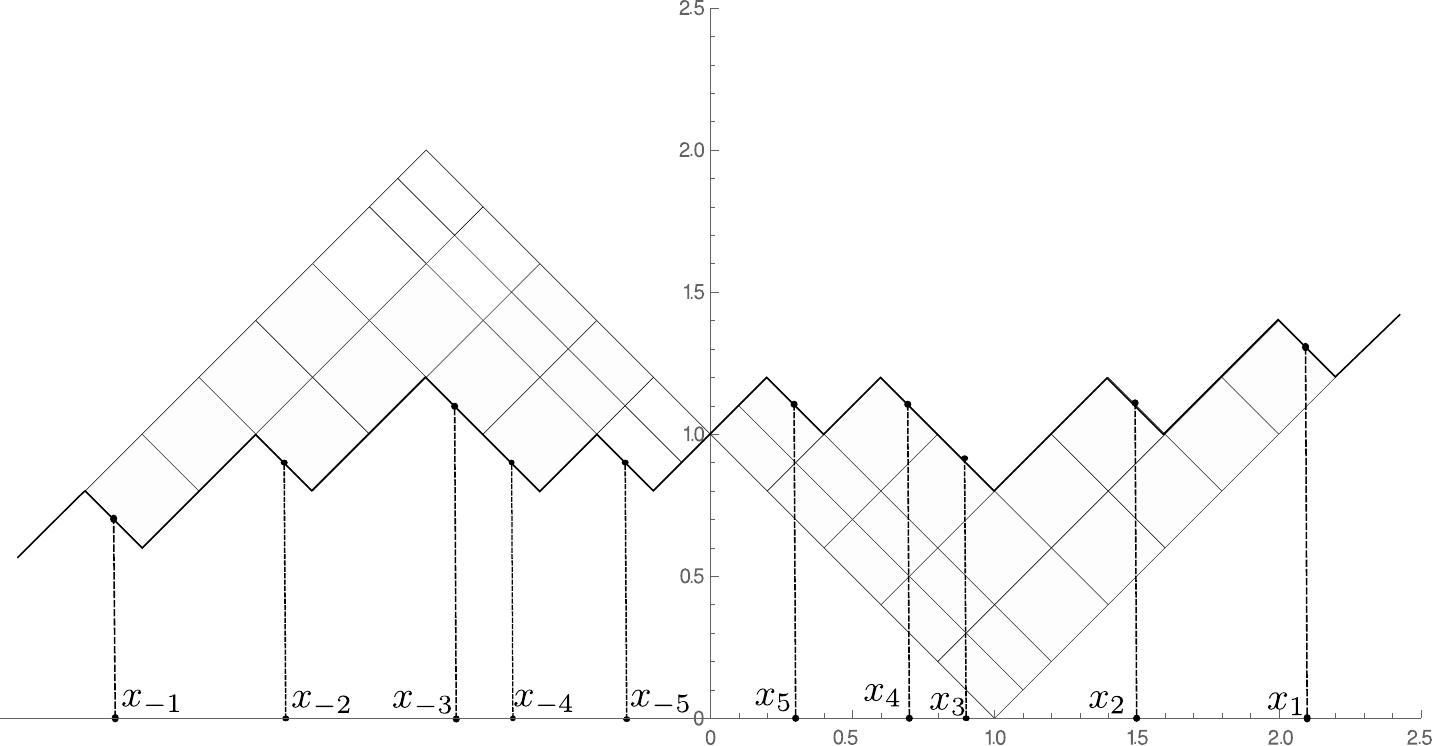}
  \caption{Rotated and scaled diagram for $n=5$ and its continuation to
    negative values of coordinate $x$. The function $f_{n}(x)$ is shown in solid black,
    the points $x_{i}=\frac{a_{i}}{2n}$ are the midpoints of intervals
    where $f_{n}'(x)=-1$. }
  \label{fig:young-rotated-and-continued}
\end{figure}

\begin{lemma}
  \label{lemma:func-form-of-prob-1}
  
  Consider a dominant integral weight $\lambda$ with the coordinates
  $\{a_{i}\}$, defined in Section \ref{sec:statements}. Then the
  probability measure $\mu_{n}(\lambda)$ can be approximated by the
  exponent of the  quadratic functional $J$ of $f_{n}$:
  \begin{multline}
    \label{eq:69}
    \mu_{n}(\lambda)    =e^{-(2n)^{2} J[f_{n}]+\mathcal{O}(n\ln n)}=\\=\exp\left(-(2n)^{2}\left[\frac{1}{2}\int_{-c/2}^{c/2}\int_{-c/2}^{c/2}
        \frac{1}{16}f_{n}'(x)f_{n}'(y)\ln|x-y|^{-1}\dx\;\dy+C\right]+\mathcal{O}(n\ln n)
    \right),
  \end{multline}
  where $f_{n}$ is an upper boundary of the diagram for $\lambda$,
  rotated and scaled as described in Section \ref{sec:statements}, and the constant $C$ is given by the formula
  \begin{equation}
    \label{eq:70}
    C=-\frac{1}{32}c^{2}\ln c+\frac{(c-2)^{2}}{16}\ln (c-2)+\frac{c-1}{4}\ln 2-\frac{3}{64}(c-4)^{2}.
  \end{equation}

\end{lemma}

\begin{proof}
  We start with the functional and demonstrate that
  integrals are the approximations to the sums in the probability measure,
  written in the form \eqref{eq:13}.

  First, we need to show that an
  integral in the functional can be written as a sum of two integrals,
  one of which contains the potential $V_{0}(x)$:
  \begin{multline}
    \label{eq:71}
    \frac{1}{16}\int_{-c/2}^{c/2}\int_{-c/2}^{c/2}
    f_{n}'(x)f_{n}'(y)\ln|x-y|^{-1}\dx\;\dy+\left(\frac{c^{2}}{16}\ln c-\frac{3c^{2}}{32}+\frac{c}{2}\right)=\\
    = \int_{-c/2}^{c/2}\int_{-c/2}^{c/2}
        \frac{1}{4}(1-f_{n}'(x))\cdot\frac{1}{4}(1-f_{n}'(y))\ln|x-y|^{-1}\dx\;\dy+2\int_{-c/2}^{c/2}\frac{1}{4}(1-f_{n}'(x))V_{0}(x)\dx.
  \end{multline}
It is easily done by expanding the brackets in the double integral on
the right hand side and using the equalities 
\begin{equation}
  \label{eq:72}
  \int_{-c/2}^{c/2}\ln|x-y|^{-1}\dy=c-\left[\left(\frac{c}{2}+x\right)\ln\left(\frac{c}{2}+x\right)+\left(\frac{c}{2}-x\right)\ln\left(\frac{c}{2}-x\right)\right]=c-4V_{0}(x),
\end{equation}
\begin{equation}
  \label{eq:73}
  \int_{-c/2}^{c/2}\int_{-c/2}^{c/2}\ln|x-y|^{-1}\dx\;\dy=\frac{3}{2}c^{2}-c^{2}\ln c,
\end{equation}
and also the equation
\begin{equation}
  \label{eq:74}
  \int_{-c/2}^{c/2}f_{n}'(x)\dx=\int_{-c/2}^{c/2}(1-4\rho_{n}(x))\dx=c-4.
\end{equation}

Then we approximate the integrals on right-hand side of the equation
\eqref{eq:71} by the sums and demonstrate that these sums are the same
as in \eqref{eq:13}. To do so we split the intervals $\left(0,\frac{c}{2}\right)$ and
$\left(-\frac{c}{2},0\right)$ into small intervals of length
$\frac{1}{n}$ (for odd $a_{n}$ the first interval should be of length
$\frac{1}{2n}$):
\begin{multline}
  \label{eq:76}
  \frac{1}{16}\sum_{k=-\lfloor cn/2 \rfloor}^{\lfloor cn/2 \rfloor}
  \sum_{l=-\lfloor cn/2 \rfloor}^{\lfloor cn/2 \rfloor} \int_{k/n}^{(k+1)/n}
  \int_{l/n}^{(l+1)/n} (1-f_{n}'(x)) (1-f_{n}'(y))\ln|x-y|^{-1}\dx\,\dy+ \\ 
  +\frac{1}{2} \sum_{k=-\lfloor cn/2 \rfloor}^{\lfloor cn/2 \rfloor}
  \int_{k/n}^{(k+1)/n} (1-f_{n}'(x))V_{0}(x)\dx.
\end{multline}
On each interval $\left(\frac{k}{n},\frac{k+1}{n}\right)$ we have
$f_{n}'(x)=\pm 1$. Recall that the coordinates of midpoints of intervals, where
$f_{n}'(x)=-1$, are $\frac{a_{l}}{2n}$.

Approximating the double integral using the first order of Taylor
expansion on each small interval with $k\neq l$, we obtain the
contribution:
\begin{multline}
  \label{eq:75}
  \frac{1}{4n^{2}}\sum_{i=-n;i\neq 0}^{n}\sum_{j=-n;i\neq j; j\neq
    0}^{n}\ln\left|\frac{a_{i}}{2n}-\frac{a_{j}}{2n}\right|^{-1}=\\=\frac{1}{4n^{2}}\sum_{i=-n;i\neq
    0}^{n}\sum_{j=-n;i\neq j; j\neq
    0}^{n}\ln\left|a_{i}-a_{j}\right|^{-1}+\ln |2n|.
\end{multline}
To estimate the correction to the sum we need to consider the
difference of the integral over one of the squares, where
$f_{n}'(x)=-1$ and the corresponding term in the sum:
\begin{equation}
  \label{eq:46}
  E_{ij}=\frac{1}{16}\int_{k/n}^{(k+1)/n}
  \int_{l/n}^{(l+1)/n} (1-f_{n}'(x))
  (1-f_{n}'(y))\ln|x-y|^{-1}\dx\,\dy-\frac{1}{4n^{2}}
  \ln\left|\frac{a_{i}}{2n}-\frac{a_{j}}{2n}\right|^{-1}.
\end{equation}
Denote $\frac{a_{i}}{2n},\frac{a_{j}}{2n}$ by $x_{i}, y_{i}$
correspondingly, then the difference is rewritten as
\begin{equation}
  \label{eq:98}
  E_{ij}=\frac{1}{4}\int_{x_{i}-1/(2n)}^{x_{i}+1/(2n)}\int_{y_{j}-1/(2n)}^{y_{j}+1/(2n)} \left(\ln|x-y|^{-1}-\ln|x_{i}-y_{j}|^{-1}\right)\dx\,\dy
\end{equation}
Changing the variables to $\tilde{x}=x-x_{i},\; \tilde{y}=y-y_{j}$, we
arrive at the integral
\begin{equation}
  \label{eq:99}
  E_{ij}=\frac{1}{4}\int_{-1/(2n)}^{1/(2n)}\int_{-1/(2n)}^{1/(2n)} \ln\left|1+\frac{\tilde{x}-\tilde{y}}{x_{i}-y_{j}}\right|^{-1}\mathrm{d\tilde{x}}\,\mathrm{d\tilde{y}}
\end{equation}
Since $|\ln(1+t)|<2|t|$ for
$t\in\left[-\frac{1}{2},\frac{1}{2}\right]$, this integral is
estimated as\linebreak
${|E_{ij}|<\frac{1}{4n^{2}}\frac{1}{2n|x_{i}-y_{j}|}=\frac{1}{4n^{2}}\frac{1}{|a_{i}-a_{j}|}}$. Thus in the sum
$\sum_{i}\sum_{j}E_{ij}$ we have at most $2n$ terms with $|i-j|=1$ that are estimated
by $\frac{1}{4n^{2}}\frac{1}{1}$, at most $2n$ terms with $|i-j|=2$ that are estimated
by $\frac{1}{4n^{2}}\frac{1}{2}$, then we have $2n$ terms that are estimated by
$\frac{1}{4n^{2}}\frac{1}{3}$ and so on. The total sum is of the order $(2n)^{2}\sum_{i,j}E_{ij}=\mathcal{O}(n\ln n)$.

For contributions with $k=l$ we use the integral
\begin{equation}
  \label{eq:77}
  \frac{1}{4}\int_{k/n}^{(k+1)/n} \int_{k/n}^{(k+1)/n}\ln|x-y|^{-1}\dx\,\dy=\frac{1}{n^{2}} \ln n+\frac{3}{2 n^2},
\end{equation}
so the total contribution of $n$ such diagonal terms with a factor $(2n)^{2}$ in the exponent of the
expression \eqref{eq:69} is of order
$\mathcal{O}(n\ln n)$. 

We approximate the integral with $V_{0}(x)$ in a similar way and obtain
\begin{equation}
  \label{eq:78}
  \frac{1}{2} \sum_{k=-\lfloor cn/2 \rfloor}^{\lfloor cn/2 \rfloor}
  \int_{k/n}^{(k+1)/n} (1-f_{n}'(x))V_{0}(x)\dx=\sum_{i=-2n,i\neq 0}^{2n}\frac{1}{n}V_{0}\left(\frac{a_{i}}{2n}\right)+\mathcal{O}\left(\frac{1}{n^{2}}\right).
\end{equation}
We substitute expressions \eqref{eq:78},\eqref{eq:77},\eqref{eq:75} into
\eqref{eq:69} and see that we get a leading $a_{i}$-dependent part of the
expression \eqref{eq:13} correctly. But we need also to compute the
$a_{i}$-independent part and compare it to the asymptotic behavior of $Z_{n}$
in $n$.  We use the formulas \eqref{eq:19},\eqref{eq:10},\eqref{eq:20}
to derive the asymptotic behavior of $Z_{n}$:  
\begin{equation}
  \label{eq:79}
  \frac{1}{Z_{n}}=\exp\left(-n\ln 2-n\ln \tilde{C}_{n}+\ln
    C_{n}\right)=\exp\left(-n\ln 2-n\ln \pi-cn^{2}\ln n+cn^{2}+\ln C_{n}\right).
\end{equation}
For $\ln C_{n}$ we have
\begin{equation}
  \label{eq:80}
  \ln C_{n}=\left[-cn^{2}+\mathcal{O}(n)\right]\ln 2+\ln (n!)-\sum_{k=1}^{n}\ln
((2k)!) + \sum_{k=1}^{n} \ln\left[(N+2k-2)!\right].
\end{equation}
To estimate factorials it is convenient to combine them in the following way:
\begin{multline}
  \label{eq:81}
  \ln\left[\prod_{k=0}^{n-1}\frac{\left(N+2k\right)!}{(2n-2k)!}\right]=\sum_{k=0}^{n-1}\sum_{j=1}^{N-2n+4k-1}\ln(2n-2k+j)=\\
  =\sum_{k=0}^{n-1}\sum_{j=1}^{N-2n+4k-1}\ln
  2n +\sum_{k=0}^{n-1}\sum_{j=1}^{(c-4)n+4k-1} \ln\left(1-\frac{k}{n}+\frac{j}{2n}\right).
\end{multline}
The first double sum gives us $\sum_{k=0}^{n-1}\sum_{j=1}^{N-2n+4k-1}\ln
  2n=(c-2)n^{2}\ln 2n-3n\ln 2n+\mathcal{O}(n)$, the second sum is a
Riemann sum for an integral
\begin{multline}
  \label{eq:82}
n^{2}\int_{0}^{1}\dx\int_{0}^{(c-4)+4x}\dy\ln\left(1-x+\frac{y}{2}\right)=\\=-n^{2}\left[\frac{1}{4}
c^2 \ln (c-2)-\frac{1}{4} c^2 \ln c+\frac{3 }{2}c-c \ln (c-2)+c \ln 2+\ln
(c-2)-3-\ln 2\right].
\end{multline}
Combining these results we can write down leading contributions to
$\ln Z_{n}$ in $n$. We preserve the terms of the order $n^{2}\ln n$
from expressions \eqref{eq:81}, \eqref{eq:79} and of order $n^{2}$
from formulas \eqref{eq:79}, \eqref{eq:80}, \eqref{eq:82}:
\begin{multline}
  \label{eq:9}
  -\ln Z_{n}=-cn^{2}\ln n+cn^{2} +(c-2)n^{2}\ln n
  +(c-2)n^{2}\ln 2-cn^{2}\ln 2\\
  -n^{2}\left[\frac{1}{4}
c^2 \ln (c-2)-\frac{1}{4} c^2 \ln c+\frac{3 }{2}c-c \ln (c-2)+c \ln 2+\ln
(c-2)-3-\ln 2\right]+\mathcal{O}(n\ln n)
\end{multline}
Let us now combine the contributions of the order $n^{2}\ln n$ from
$\ln Z_{n}$ and from the equation \eqref{eq:75}. Substituting the
equation \eqref{eq:75} into \eqref{eq:69}, we get
$\frac{1}{2}(2n)^{2}\ln |2n|=2n^{2}\ln n+2n^{2}\ln 2$. Thus we see
that terms of the order $n^{2}\ln n$ are cancelled in the equation
\eqref{eq:69}.

Now combining the contributions of order $n^{2}$ from
\eqref{eq:71},\eqref{eq:9},\eqref{eq:75}, we obtain the
expression for a constant $C$:
\begin{multline}
  \label{eq:83}
  C=\frac{1}{32}c^{2}\ln c-\frac{3}{64}c^{2}+\frac{1}{4}c-\frac{1}{2}\ln 2+\\
  \frac{1}{2}\ln 2 -\frac{c}{4}-\frac{1}{16}c^{2}\ln c
  +\frac{1}{16}c^{2}\ln(c-2)-\frac{1}{4}(c-1)\ln(c-2)+\frac{c}{4}\ln 2 -\frac{1}{4}\ln 2
  -\frac{3}{4}+\frac{3}{8}c=\\
  =-\frac{1}{32}c^{2}\ln c+\frac{(c-2)^{2}}{16}\ln (c-2)+\frac{c-1}{4}\ln 2-\frac{3}{64}(c-4)^{2}
\end{multline}
Substituting the equations \eqref{eq:78},\eqref{eq:75} into
\eqref{eq:71}, and combining it with the expression for $C$, we see
that equation \eqref{eq:69} indeed reproduces the
probability measure \eqref{eq:13} in the leading order in $n$. 
\end{proof}

\subsection{Minimizer of the functional}
\label{sec:minimizer-functional}

The functional $J[f]$ is clearly quadratic. Rewriting the functional
$J[f]$ in terms of densities and searching for its minimum, we arrive
at the following variational problem. The particle density $\rho(x)$
that is related to the limit shape $f(x)$ by $f'(x)=1-4\rho(x)$
\eqref{eq:61}, is the minimizer of the functional
  \begin{multline}
    \label{eq:7}
    \frac{1}{2}\int_{-\frac{c}{2}}^{\frac{c}{2}}\int_{-\frac{c}{2}}^{\frac{c}{2}}
    \rho(x)\rho(y)\ln|x-y|^{-1}\dx\,\dy + \\
    +\frac{1}{4}\int_{-\frac{c}{2}}^{\frac{c}{2}} \rho(x)
    \left[\left(\frac{c}{2}+x\right)\ln\left(\frac{c}{2}+x\right)+\left(\frac{c}{2}-x\right)\ln\left(\frac{c}{2}-x\right)\right]\dx.
  \end{multline}

  We first recall necessary and sufficient conditions for the
  minimizer (Proposition \ref{theorem:deift-min-suff-cond}). Our
  functional is strictly convex (see \cite{deift1999orthogonal},
  Theorem 6.27), therefore the minimizer is unique.
  \begin{prop}[\cite{deift1999orthogonal}, Thm 6.132]
    \label{theorem:deift-min-suff-cond}
    Suppose $\rho(x)$ is a continuous function on $[-c/2, c/2]$. Then
    $\rho(x)$ is the minimizer of \eqref{eq:7} if and only if there
    exists a constant $\ell\in\mathbb{R}$ such that
  \begin{enumerate}[(I)]
  \item
    \begin{equation*}
      \int_{-\frac{c}{2}}^{\frac{c}{2}}\ln|x-y|^{-1}\rho(y)\dy +V_{0}(x) \ge \ell\quad \text{for any}\quad x\in \left[-\frac{c}{2}, \frac{c}{2}\right],
    \end{equation*}
  \item
    \begin{equation*}
      \int_{-\frac{c}{2}}^{\frac{c}{2}}\ln|x-y|^{-1}\rho(y)\dy +V_{0}(x) = \ell \quad\text{for}\quad x\in\supp\rho.
    \end{equation*}
  \end{enumerate}
  \end{prop}

  We construct the minimizer by an explicit integral formula, which is
  obtained as a solution of Riemann-Hilbert problem, as described in
  the book by P.Deift \cite{deift1999orthogonal}. This solution is
  presented in Lemma \ref{lemma:rho-for-c-greater-4} below, where we
  also check the necessary and sufficient conditions.
  \begin{lemma}
    \label{lemma:rho-for-c-greater-4}
  For $c\ge 4$
  \begin{equation}
    \label{eq:50}
  \rho(x)=\frac{1}{\pi^{2}}\Re\left[\sqrt{2c-4-x^2}\int_{-\sqrt{2c-4}}^{\sqrt{2c-4}}\frac{\frac{1}{4}\left(\ln\left(\frac{c}{2}+s\right)-\ln\left(\frac{c}{2}-s\right)\right)}{\sqrt{2c-4-s^{2}}(s-x)}\mathrm{ds}\right],
\end{equation}
is the minimizer of the functional \eqref{eq:7}. 
\end{lemma}
\begin{proof}
  We want to find a minimizer of \eqref{eq:7}. We assume that the
  minimizer is a continuous density $\rho(x)$ and $\rho(x)=\rho(-x)$
  due to condition \eqref{eq:12}. Thus we can assume that
  $\mathrm{supp}\;\rho=[-a,a]$ with $a<\frac{c}{2}$. Additionally,
  since total number of particles is $2n$, we have
  $\int_{-c}^{c}\rho_{n}(x)=1$, so the function $\rho$ should satisfy
  the condition
  \begin{equation}
    \label{eq:25}
    \int_{-a}^{a}\rho(x)\;\dx=1.
  \end{equation}

 Then we have a variational problem with the convex functional $\mathcal{E}[\rho]$:
  \begin{equation}
    \label{eq:22}
    \mathcal{E}[\rho]=\int_{-a}^{a}\int_{-a}^{a}\left(-\frac{1}{2}\ln|x-y|+\frac{1}{2}V_{0}(x)+\frac{1}{2}V_{0}(y)\right)\rho(x)\rho(y)\;\dx\;\dy.
  \end{equation}

  Taking variation by $\rho$, we obtain Euler-Lagrange equation for $x\in\supp\rho$:
  \begin{equation}
    \label{eq:23}
    \int_{-a}^{a}\ln|x-y|^{-1}\rho(y)\dy +V_{0}(x)=\mathrm{const}. 
  \end{equation}
  This equation has a natural electrostatic interpretation, since $\ln|x-y|^{-1}$ is a two-dimensional
  electrostatic potential at the point $x$ of a unit charge at $y$. Taking the
  derivative with respect to $x$ of the equation \eqref{eq:23}, we get the equilibrium
  condition for the charge distribution with density $\rho$ in an external field with
  the potential $V_{0}$:
  \begin{equation}
    \label{eq:24}
    -\int_{-a}^{a}\frac{\rho(y)\;\dy}{y-x} +V_{0}'(x)=0.
  \end{equation}

  The charge distribution $\rho(x)$ creates a potential $\varphi(z)$
  on a complex plane with a cut from $-a$ to $a$ along the real axis
  $\mathbb{C}\setminus [-a,a]$. The potential $\varphi(z)$ satisfies
  Laplace equation and is a harmonic function. Its gradient is called
  field strength and thus its components along the real and the
  imaginary axes are analytic functions on
  $\mathbb{C}\setminus [-a,a]$. Denote by $G(z)$ the field strength
  component along the imaginary axis. Then
  \begin{equation}
    \label{eq:29}
     G(z)=-i\int_{-a}^{a}\frac{\rho(y)}{y-z}\dy,
  \end{equation}
  is a Hilbert transform of $\rho(x)$.
We know that  $G(z)$ is analytic on $\mathbb{C}\setminus
[-a,a]$ and its limit values are
\begin{multline}
  \label{eq:30}
  G_{\pm}(x)=
  \lim_{\varepsilon\to 0}\frac{1}{i}\int\frac{\rho(y)\dy}{y-(x\pm i\varepsilon)}=
  \lim_{\varepsilon\to 0}\frac{1}{i}\int\frac{y-x\pm i\varepsilon}{(y-x)^2+\varepsilon}\rho(y)\dy=\\
  =-i\mathrm{p.v.}\int\frac{\rho(y)\dy}{y-x}\pm\pi\rho(x),
\end{multline}
where we have used
$\frac{\varepsilon}{\pi(x^2+\varepsilon^2)}\to\delta(x)$.
Thus we arrive at
\begin{equation}
\label{eq:31}
  G_{\pm}(x)=\pm\pi\rho(x)+iV_{0}'(x),
\end{equation}
so on the support of $\rho(x)$ we have
\begin{equation}
  \label{eq:32}
  G_+(x)+G_-(x)=2iV_{0}'(x),\quad x\in[-a,a],
\end{equation}
and outside of $[-a,a]$ the following conditions appear
\begin{equation}
  \label{eq:33}
  \begin{array}{l}
    G_+(x)-G_-(x)=0,\quad x\not\in[-a,a],\\
    G(z)\to 0 ,\; z\to\infty.
  \end{array}    
\end{equation}

This is Riemann-Hilbert problem for $G(z)$, but the equation
(\ref{eq:32}) is in a non-standard form with the sum instead of a
difference. We need to redefine $G$ in such a way  as to obtain a
standard problem that can be solved by the Plemelj formula
\cite{deift1999orthogonal}:
\begin{equation}
\label{eq:34}
  \tilde{G}(z)=\frac{G(z)}{\sqrt{z^2-a^2}}.
\end{equation}
Then we get
\begin{multline}
  \tilde{G}_+(x)-\tilde{G}_-(x)=\frac{G_+(z)}{\left(\sqrt{x^2-a^2}\right)_+}-\frac{G_-(z)}{\left(\sqrt{x^2-a^2}\right)_-}=\\
  =\frac{G_+(z)+G_-(z)}{\left(\sqrt{x^2-a^2}\right)_+}=\frac{2iV_{0}'(x)}{\left(\sqrt{x^2-a^2}\right)_+},
\end{multline}
where the branch of the square root changes the sign crossing the real
line
\begin{equation}
  \label{eq:35}
\left(\sqrt{x^2-a^2}\right)_+=-\left(\sqrt{x^2-a^2}\right)_-,\quad x\in[-a,a].
\end{equation}
The condition \eqref{eq:33} is preserved for $\tilde{G}$:
\begin{equation}
  \label{eq:36}
  \begin{array}{l}
  \tilde{G}_+(z)-\tilde{G}_-(z)=0,\quad z\not\in[-a,a],\\
    \tilde{G}(z)\to 0,\quad z\to\infty.
  \end{array}    
\end{equation}

Then $\tilde{G}(z)$ is a solution of the standard Riemann-Hilbert
problem and is given by the Plemelj formula
\begin{eqnarray}
\label{eq:37}
  \tilde{G}(z)=\frac{1}{2\pi
  i}\int_{-a}^a\frac{2
  iV_{0}'(s)\mathrm{ds}}{\left(\sqrt{s^2-a^2}\right)_+(s-z)},\\
  G(z)=\frac{\sqrt{z^2-a^2}}{
  \pi }\int_{-a}^a\frac{V_{0}'(s)\mathrm{ds}}{\left(\sqrt{s^2-a^2}\right)_+(s-z)}.
\end{eqnarray}
To find the support of $\rho$ we need to consider the asymptotics of
$G(z)$ as $z\to\infty$. We expand the above
expression into series:
\begin{equation}
  \label{eq:38}
  G(z)=\frac{z+\dots}{
    \pi}\left(-\frac{1}{z}\right)\int_{-a}^a\frac{V_{0}'(s)}{\left(\sqrt{s^2-a^2}\right)_+}\left(1+\frac{s}{z}+\dots\right)\mathrm{ds}.
\end{equation}
Consider the first term in the series, for $G(z)\to 0, z\to\infty$ we need to have
\begin{equation}
\label{eq:39}
  \int_{-a}^a\frac{V_{0}'(s)}{\left(\sqrt{s^2-a^2}\right)_+}\mathrm{ds}=0.
\end{equation}
This condition is automatically satisfied, since $V_{0}(x)$ is an even
function and $V_{0}'(s)$ is an odd function. At the same time 
\begin{equation}
\label{eq:40}
  G(z)=-i\int\frac{\rho(y)\dy}{y-z}\simeq\frac{i}{z}\int\rho(y)\dy+{\cal O}\left(\frac{1}{z^2}\right),
\end{equation}
and comparing it to the second term in the series \eqref{eq:38} we
arrive at
\begin{equation}
\label{eq:41}
  -\frac{1}{
    \pi}\int_{-a}^{a}\frac{V_{0}'(s)s}{\left(\sqrt{s^2-a^2}\right)_+z}=\frac{i}{z}.
\end{equation}
Taking the derivative of the equation \eqref{eq:17} and
substituting to the above equation, we get
\begin{equation}
  \label{eq:43}
  \frac{1}{4}\int_{-a}^{a}\frac{s}{\sqrt{s^2-a^2}}\cdot\frac{-1}{\pi}\ln\left|\frac{s+c/2}{s-c/2}\right|\mathrm{ds}=i.
\end{equation}
By taking a derivative we can check that
\begin{multline}
  \label{eq:65}
  \int\frac{s}{\sqrt{s^2-a^2}}\ln\left|\frac{s+c/2}{s-c/2}\right|\ds=
  \frac{1}{2} \left(\left(2 \sqrt{s^2-a^2}-\sqrt{c^2-4 a^2}\right) \log (c-2 s)+\right.\\\left.+\left(\sqrt{c^2-4 a^2}-2 \sqrt{s^2-a^2}\right) \log (c+2
   s)-\sqrt{c^2-4 a^2} \log \left(\sqrt{c^2-4 a^2} \sqrt{s^2-a^2}-2 a^2-c s\right)+\right.\\\left.+\sqrt{c^2-4 a^2} \log \left(\sqrt{c^2-4 a^2}
   \sqrt{s^2-a^2}-2 a^2+c s\right)-2 c \log \left(\sqrt{s^2-a^2}+s\right)\right)+\mathrm{const}.
\end{multline}
Substituting the integration limits we obtain the equation
\begin{equation}
  \label{eq:47}
  \frac{c}{4}\left[1-\sqrt{1-\left(\frac{2a}{c}\right)^{2}}\right]=1,
\end{equation}
which can be solved for $c\geq 4$ and we get
\begin{equation}
  \label{eq:48}
  a=\sqrt{2c-4}.
\end{equation}
We see that indeed $a<\frac{c}{2}$ and the solution  $\rho$ of the
variational problem \eqref{eq:7} is given
by the formula
\begin{equation}
 \rho(x)=\frac{1}{\pi}\Re[ G_+(x)]=\frac{1}{\pi^{2}}\Re\left[\sqrt{x^2-2c+4}\int_{-\sqrt{2c-4}}^{\sqrt{2c-4}}\frac{\frac{1}{4}\left(\ln\left(\frac{c}{2}+s\right)-\ln\left(\frac{c}{2}-s\right)\right)}{\left(\sqrt{s^2-2c+4}\right)_+(s-x)}\mathrm{ds}\right].
\end{equation}
By choosing the proper branch of the square roots we obtain formula
\eqref{eq:50}. 

We constructed $\rho(x)$ in such a way that that minimization
condition (II) of the Proposition \ref{theorem:deift-min-suff-cond} is
satisfied. Using the condition (II) for we rewrite the  condition (I) for $x>a$ as
\begin{equation}
  \label{eq:106}
  \int_{a}^{x}(\rho(y)\ln|x-y|-V_{0}(y)) \dy \leq 0.
\end{equation}

To check this condition, we first rewrite the integral expression
\eqref{eq:37} for $G(z)$ as a contour integral
\begin{equation}
  \label{eq:100}
  G(z)=\frac{\sqrt{z^{2}-a^{2}}}{2\pi} \int_{\mathcal{C}} \frac{V_{0}'(s)\ds }{\left(\sqrt{s^{2}-a^{2}}\right)(s-z)}, 
\end{equation}
where $\mathcal{C}$ is a clockwise contour with $[-a,a]$ in
its interior, that lies in the domain of analyticity of $V_{0}$. To
compute this integral we use the power series with the positive
coefficients
\begin{equation}
  V_{0}'(s)=\sum_{m=1}^{\infty}\frac{2^{2(m-1)}}{(2m-1)c^{2m-1}} s^{2m-1}.
\end{equation}
We consider the contour $\mathcal{C}$ as a contour around the points $z$
and $\infty$ and compute the integral of $s^{2m-1}$ by the residue calculation as in
\cite{deift1999orthogonal} eq. (6.150) to obtain
\begin{equation}
  \label{eq:102}
  \frac{\sqrt{z^{2}-a^{2}}}{2\pi} \int_{\mathcal{C}} \frac{s^{2m-1}\ds }{\left(\sqrt{s^{2}-a^{2}}\right)(s-z)}=iz^{2m-1}-i\sqrt{z^{2}-a^{2}}\left(z^{2m-2}+\sum_{j=1}^{m-1}z^{2m-2-2j} a^{2j}\prod_{l=1}^{j}\frac{2l-1}{2l}\right). 
\end{equation}
Combining all the terms in the series, we get
\begin{equation}
  \label{eq:103}
  G(z)=iV_{0}'(z)-i\sqrt{z^{2}-a^{2}}h(z), 
\end{equation}
where $h(z)$ is an analytic function, which has the series with
strictly positive coefficients. Then for $|x|<a$ we have
$\rho(x)=\Re \left[G_{+}(x)\right]=-i\left(\sqrt{x^{2}-a^{2}}\right)_{+}h(x)$. For
$x>a$ we substitute the equations \eqref{eq:29},\;\eqref{eq:103} to the
formula \eqref{eq:106} and arrive at the inequality
\begin{equation}
  \label{eq:105}
  -\int_{a}^{x}\sqrt{y^{2}-a^{2}}h(y)\dy\leq 0, 
\end{equation}
which holds, since the function under the integral is positive.
\end{proof}

For $c<4$ we can no longer solve the variational problem with a smooth
function $\rho$ such that
$\supp\rho\subset\left[-\frac{c}{2},\frac{c}{2}\right]$. 
The potential $V_0(x)$ becomes weaker when $c$ tends to $4$, and when $c=4$ we have a phase transition. In the latter case particles are not confined strictly inside the interval $[-2,2]$ anymore, instead we have a constant density $\rho(x)\equiv 1/4$ on the whole interval. We have an obvious restriction $\rho(x) \leq 1/2$, therefore for $c<4$ it is reasonable to expect 
\begin{equation}
    \label{eq:49}
\rho(x)=\frac{1}{2} - \rho_{1}(x),  
\end{equation}
where $\supp \rho_1 \subset \left[-\frac{c}{2}, \frac{c}{2}\right]$ (see Figure~\ref{fig:Bn-limit-shapes-3}). For $c\ge 4$ it is possible to use Proposition \ref{theorem:deift-min-suff-cond} to describe the minimizer only because the condition $\rho(x) < 1/2$ is satisfied in this case. For $c < 4$ it is no longer the case and we restate the criterion for the minimizer in the following way (with basically the same proof).
  \begin{prop}
    \label{theorem:min-suff-cond-2}
  Suppose $\rho(x)=\frac{1}{2} - \rho_{1}(x)$ is a continuous function
  on $\left[-\frac{c}{2}, \frac{c}{2}\right]$, and $\frac{1}{2} >
  \rho_{1}(x) \ge 0$. Then
  $\rho(x)$ is the minimizer of \eqref{eq:7} if and only if there exists a constant $\ell\in\mathbb{R}$ such that
  \begin{enumerate}[(I)]
      \item
        \begin{equation*}
          \int_{-\frac{c}{2}}^{\frac{c}{2}}\ln|x-y|^{-1}\rho_1(y)\dy +V_{0}(x) \geq \ell
          \quad \text{for any} \quad x\in \left[-\frac{c}{2}, \frac{c}{2}\right],
        \end{equation*}

      \item
        \begin{equation*}
                    \int_{-\frac{c}{2}}^{\frac{c}{2}}\ln|x-y|^{-1}\rho_1(y)\dy +V_{0}(x) = \ell
          \quad \text{for} \quad x\in \mathrm{supp}\;\rho_{1}.
        \end{equation*}
  \end{enumerate}
  \end{prop}

In the next lemma we show that
$\rho_{1}(x)$ can be computed the same way as in Lemma
\ref{lemma:rho-for-c-greater-4} and solve the variational problem for 
$c\in [2;4]$.

\begin{lemma}
\label{lemma:rho-for-c-less-than-4}
  For $2\leq c\leq 4$
\begin{equation}
  \label{eq:8}
  \rho(x)=\frac{1}{2}-\frac{1}{\pi^{2}}\Re\left[\sqrt{x^2-2c+4}\int_{-\sqrt{2c-4}}^{\sqrt{2c-4}}\frac{\frac{1}{4}\left(\ln\left(\frac{c}{2}+s\right)-\ln\left(\frac{c}{2}-s\right)\right)}{\left(\sqrt{s^2-2c+4}\right)_+(s-x)}\mathrm{ds}\right].
\end{equation}

\end{lemma}
\begin{proof}
  Denote by $\rho_{0}(x)\equiv \frac{1}{4}$ a constant solution to the
  equation  \eqref{eq:24}. If we look for a solution 
  to the variational problem
  \eqref{eq:7} with the normalization condition \eqref{eq:25} 
  in the form \eqref{eq:49}
  so that 
  \begin{equation}
    \label{eq:97}
    \int_{-c/2}^{c/2}\frac{\rho(y)\dy}{x-y}=\int_{-c/2}^{c/2}\frac{(2\rho_{0}(y)-\rho_{1}(y))\dy}{x-y}=-2V_{0}'(x)+\int_{-c/2}^{c/2}\frac{\rho_{1}(y)\dy}{x-y}=-V_{0}'(x),
  \end{equation}
  then
  the function $\rho_{1}(x)$ should also be  a solution of
  \eqref{eq:24}, but 
  with a different normalization condition
  \begin{equation}
    \label{eq:52}
    \int_{-c/2}^{c/2}\rho_{1}(x)\dx=-\int_{-c/2}^{c/2}\rho(x)\dx+2\int_{-c/2}^{c/2}\rho_{0}(x)\dx=\frac{c-2}{2}.
  \end{equation}
  The integral representation of $\rho_{1}(x)$ is obtained 
  in the same way as in Lemma \ref{lemma:rho-for-c-greater-4}, but equation \eqref{eq:47} becomes
  \begin{equation}
    \label{eq:53}
     \frac{c}{4}\left[1-\sqrt{1-\left(\frac{2a}{c}\right)^{2}}\right]=\frac{c-2}{2},
  \end{equation}
  and we again get $a=\sqrt{2c-4}$.
  We see that $\rho_{1}(x)$ is given by the formula \eqref{eq:50} and
  satisfies the normalization condition \eqref{eq:52}. Thus we have derived the
  expression \eqref{eq:8} for $\rho(x)$ for $c\in[2,4]$.
  The check of the  minimization conditions (I) and (II) of Proposition
  \ref{theorem:min-suff-cond-2} is the same as in Lemma
  \ref{lemma:rho-for-c-greater-4}. 

\end{proof}

\begin{figure}[h]
  \centering
  \includegraphics[width=10cm]{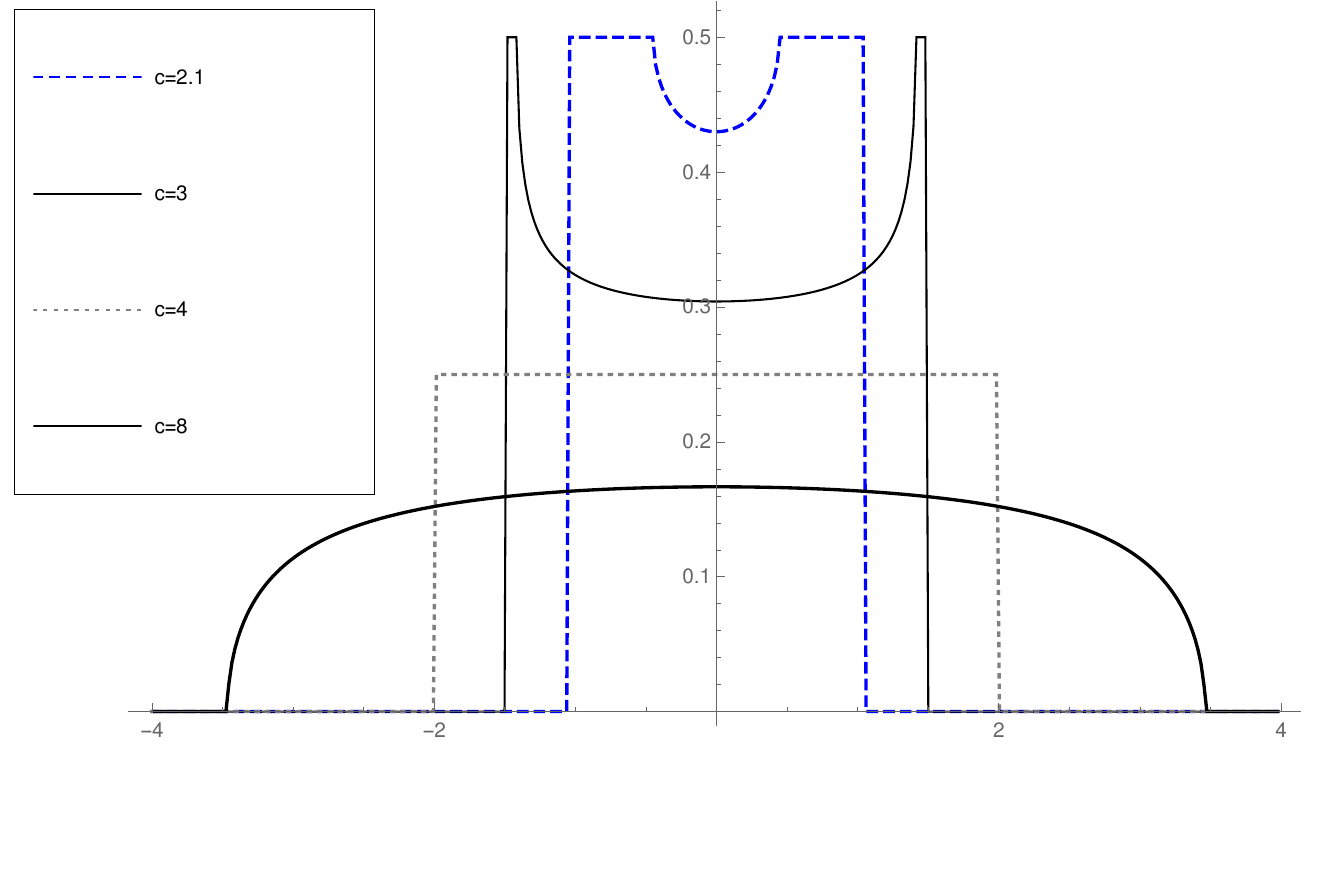}
  \caption{Plots of the density $\rho(x)$ given by the formula~\eqref{eq:68} for $c=2.1$ (blue dash), $c=3$ (black line) $c=4$ (gray dots) $c=8$ (thick black line). }
  \label{fig:density-plots}
\end{figure}

The integral in the formula \eqref{eq:50} can be calculated and we can
write down the
expression for $\rho(x)$ in terms of inverse trigonometric functions.
\begin{lemma}
  \label{lemma:explicit-limit-for-rho}
  The density $\rho(x)$ can be written explicitly as:

  \begin{equation}
    \label{eq:54}
    \rho(x)=\left\{
    \begin{array}{l}
    \frac{\theta(\sqrt{2c-4}-|x|)}{4\pi}\left[\arctan \left(\frac{-c (x-4)-8}{(c-4) \sqrt{2
      c-4-x^2}}\right)+\arctan\left(\frac{c (x+4)-8}{(c-4) \sqrt{2
      c-4-x^2}}\right)\right], \quad c\geq 4,\\
     \frac{1}{2}-\frac{\theta(\sqrt{2c-4}-|x|)}{4\pi}\left[\arctan \left(\frac{-c (x-4)-8}{(4-c) \sqrt{2
      c-4-x^2}}\right)+\arctan\left(\frac{c (x+4)-8}{(4-c) \sqrt{2
      c-4-x^2}}\right)\right], \; c\in[2, 4],
    \end{array} \right.
  \end{equation}
  where $\theta(\sqrt{2c-4}-|x|)$ is the Heaviside step function. By
  the use of trigonometric identities formula \eqref{eq:54} can be
  rewritten as \eqref{eq:68}.
\end{lemma}
\begin{proof}
  We need to compute the integral in \eqref{eq:50}, we can combine
  logarithms the same way as we did in equation \eqref{eq:43}:
  \begin{equation}
    \label{eq:55}
    \frac{1}{\pi^{2}}\int_{-\sqrt{2c-4}}^{\sqrt{2c-4}}\frac{\left(\ln\left(\frac{c}{2}+s\right)-\ln\left(\frac{c}{2}-s\right)\right)}{\sqrt{2c-4-s^{2}}(s-x)}\mathrm{ds}=
  \frac{1}{\pi}\int_{-\sqrt{2c-4}}^{\sqrt{2c-4}}\frac{1}{\sqrt{2c-4-s^2}(s-x)}\cdot\frac{1}{\pi}
  \ln\left|\frac{s-\frac{c}{2}}{s+\frac{c}{2}}\right|\mathrm{ds}.
  \end{equation}
To calculate this integral notice that the function
$\frac{1}{\pi}\ln\left|\frac{s-c/2}{s+c/2}\right|$ is the Hilbert
transform of the 
indicator function $\bf{1}_{[-c/2,c/2]}$.
Then we can use
the following well-known relation (see, for example, \cite{giang2010finite}):
\begin{equation}
  \label{eq:42}
  \int_{-\infty}^{\infty}f(s)\tilde{g}(s)\ds=-\int_{-\infty}^{\infty}\tilde{f}(s)g(s)\ds,
\end{equation}
where $\tilde{f}$ is a Hilbert transform of $f$ and $f\in
L^{p}(\mathbb{R})$, $g\in L^{q}(\mathbb{R})$ with
$\frac{1}{p}+\frac{1}{q}=1$. In particular, substituting 
the 
indicator function $g=\bf{1}_{[-c/2,c/2]}$ we obtain
\begin{equation}
 \label{eq:44}
  \frac{1}{\pi}\int_{-\infty}^{\infty}f(s)\ln\left|\frac{s-c/2}{s+c/2}\right|\ds=-\int_{-c/2}^{c/2}\tilde{f}(s)\ds.
\end{equation}
Thus, we need to compute the Hilbert transform for the function
${f(y)=\frac{1}{\pi}\frac{1}{\sqrt{y^{2}-2c+4}(y-x)}}$ for $y\in[-\sqrt{2c-4},\sqrt{2c-4}]=[-a,a]$, and
$f(y)=0$ for $y\not\in[-a,a]$ and then integrate it from
$-c/2$ to $c/2$. Integral in the Hilbert transform is computed explicitly by the change of variables
\begin{equation}
  \label{eq:58}
y=a\frac{a^{2}-t^{2}}{a^{2}+t^{2}},\quad
\frac{\dy}{\sqrt{a^{2}-y^{2}}}=-\frac{2a\dt}{a^{2}+t^{2}},
\end{equation}
and we obtain
\begin{equation}
  \label{eq:45}
  \tilde{f}(z)=\frac{1}{\pi^{2}}\int_{-a}^{a}\frac{ds}{\sqrt{a^{2}-s^{2}}(s-x)(s-z)}=\frac{1}{\pi}
  \frac{
    \left(\frac{1}{\sqrt{z^2-a^2}}-\frac{1}{\sqrt{x^2-a^2}}\right)}{x-z}.
\end{equation}
Now we need to compute the integral
\begin{multline}
  \label{eq:56}
  \rho(x)=\frac{1}{\pi}\Re\left[\sqrt{a^{2}-x^{2}}\int_{-c/2}^{c/2}\frac{1}{4}
    \left(\frac{1}{(x-z)\sqrt{z^2-a^2}}-\frac{1}{(x-z)\sqrt{x^2-a^2}}\right)
  dz\right].
\end{multline}
Here again we can use the substitution \eqref{eq:58} or find the
indefinite integral in the reference \cite{gradshteyn2014table} and obtain
\begin{multline}
  \label{eq:59}
  \rho(x)=-\frac{1}{4
   \pi }\left[\Im\left(\log \left(\sqrt{(c-4)^{2}} \sqrt{x^2-2 c+4}-c
       (x-4)-8\right)+
   \right.\right.\\\left.\left.
     +\log \left(\sqrt{(c-4)^{2}} \sqrt{x^2-2 c+4}+c (x+4)-8\right)\right)-\pi \right].
\end{multline}
This answer is easily rewritten in terms of the inverse trigonometric
functions for $c\geq 4, |x|\leq \sqrt{2c-4}$  as
\begin{equation}
  \label{eq:60}
  \rho(x)=\frac{1}{4\pi}\left[\arctan \left(\frac{-c (x-4)-8}{(c-4) \sqrt{2
      c-4-x^2}}\right)+\arctan\left(\frac{c (x+4)-8}{(c-4) \sqrt{2
      c-4-x^2}}\right)\right].
\end{equation}
Using Lemma \ref{lemma:rho-for-c-less-than-4}, choosing the positive
square root of $\sqrt{(c-4)^{2}}$ and taking into account the values
of the imaginary part of logarithms in formula \eqref{eq:59} for
$x^{2}>2c-4$, we obtain the desired formula \eqref{eq:54}.

Using the identities $\arctan x+\arctan y=\arctan\frac{x+y}{1+xy}$ and
$\cos
\frac{\theta}{2}=\pm\sqrt{\frac{1}{2}\pm\frac{1}{\sqrt{1+\tan^{2}\theta}}}$
and taking care of the signs we arrive at the final formula
\eqref{eq:68}.
\end{proof}

Thus we have solved the variational problem for the limit shape. The
graphs of the density are presented in Fig~\ref{fig:density-plots}. In
the next section we prove the convergence of the generalized Young
diagrams to the limit shape.

\section{Convergence of the probability measure}
\label{sec:conv-prob-meas}
In this section we use the functional and the limit shape to introduce a
pseudo-distance on the space of functions with bounded derivative. Then we
estimate the probability of weights with a given deviation according to
this distance and use this estimate to show that probability of deviation goes
to zero as $n$ goes to infinity. We use the fact that quadratic part of the
functional is the same as in the case of Vershik-Kerov-Logan-Shepp limit shape
to conclude that convergence with respect to the pseudo-distance entails the
convergence in the supremum norm.

We've proven that the probability of a weight is given by a
quadratic functional of a rotated diagram boundary $f_{n}$ (See Lemma \ref{lemma:func-form-of-prob-1}). We use the
notation that is very similar to Dan Romik's book
\cite{romik2015surprising}, the functional is denoted by $J[f_{n}]$,
its quadratic in the derivative of $f_{n}$ part is denoted by
$Q[f_{n}]$:
\begin{equation}
  \label{eq:84}
  J[f_{n}]=Q[f_{n}]+C,\quad Q[f_{n}]=\frac{1}{2}\int_{-c/2}^{c/2}\int_{-c/2}^{c/2}
        \frac{1}{16}f_{n}'(x)f_{n}'(y)\ln|x-y|^{-1}\dx\;\dy.
\end{equation}
Since our definition of $Q$ differs from a definition in the
book~\cite{romik2015surprising} only by a factor $\frac{1}{16}$, we
can use the Proposition 1.15 there and see that $Q$ is
positive-definite on compactly-supported Lipschitz functions.

Then for a compactly supported Lipschitz function $f:\mathbb{R}\to
[0,\infty)$ the quadratic part $Q$ of the functional $J$
is used to introduce a norm
\begin{equation}
  \label{eq:85}
  ||f||_{Q}=Q[f]^{1/2}.
\end{equation}
Consider a space of $1$-Lipschitz functions $f_{1},f_{2}$, such that the
derivative 
$f_{1,2}'(x)=1$ 

for $|x|>\frac{c}{2}$. Then the
difference $f_{1}-f_{2}$ is a compactly supported Lipschitz function
and we can use its norm to introduce a metric $d_{Q}$
\begin{equation}
  \label{eq:86}
  d_{Q}(f_{1},f_{2})=||f_{1}-f_{2}||_{Q}.
\end{equation}
We can use Lemma 1.21 in \cite{romik2015surprising}
to obtain an estimate on the supremum norm for a Lipschitz function $f$ with a
compact support:
\begin{equation}
  \label{eq:87}
  ||f||_{\infty}=\sup_{x}|f(x)|\leq C_{1} Q[f]^{1/4},
\end{equation}
where $C_{1}$ is some constant.

\begin{lemma}
  \label{lemma:functional-greater-than-zero}
  The value of the functional $J[f]$ on the limit shape $f$ is
  non-negative:
  \begin{equation}
    \label{eq:21}
    J[f]=\frac{1}{2}\int_{-c/2}^{c/2}\int_{-c/2}^{c/2}
        \frac{1}{16}f_{n}'(x)f_{n}'(y)\ln|x-y|^{-1}\dx\;\dy+C\geq 0,
      \end{equation}
      where $C$ is given by the equation \eqref{eq:70}. 
\end{lemma}
\begin{proof}
  We prove this lemma by constructing a sequence of weights
  $\{\lambda_{n}\}$ with the corresponding diagram boundaries $f_{n}$
  such that $\{J[f_{n}]\}$ converges to $J[f]$ and
  $\mu_{n}(\lambda_{n})$ converges to $e^{-(2n)^{2}J[f]}$. Since
  $\mu_{n}(\lambda_{n})\leq 1$ we will see that $J[f]\geq 0$.

  For $n>0$ take a weight $\lambda_{n}$ such that its diagram has $n$
  rows and the diagram boundary $f_{n}(x)\leq f(x)$ and
  $f(x)-f_{n}(x)\leq \frac{\sqrt{2}}{n}$, i.e. take the diagram with
  the longest possible rows that do not cross the limit shape. 

  Denote by $\delta f_{n}$ a difference of $f_{n}$ and $f$:
  \begin{equation}
    \label{eq:89}
    \delta f_{n}(x)=f_{n}(x)-f(x).
  \end{equation}
The function $\delta f_{n}(x)$ is a Lipschitz with a compact support.
Then from Lemma \ref{lemma:func-form-of-prob-1} we can write the probability
$\mu_{n}(\lambda)=\exp\left(-(2n)^{2}J[f_{n}]+\mathcal{O}(n\ln n)\right)$
as
  \begin{multline}
    \label{eq:90}
    \mu_{n}(\delta f_{n})=\exp\left(-(2n)^{2} \left(
        J[f]+Q[\delta f_{n}] +\frac{1}{16} \int_{-\frac{c}{2}}^{\frac{c}{2}}
        \int_{-\frac{c}{2}}^{\frac{c}{2}} f'(x) \ln
        |x-y|^{-1} \delta f_{n}'(y)
        \dx\,\dy\right)\right)\\\cdot\exp\left(\mathcal{O}(n\ln n)\right).
  \end{multline}
Denote the last term in the exponent by $L[\delta f_{n}]$:
\begin{equation}
  \label{eq:94}
  L[\delta f_{n}]:=\frac{1}{16} \int_{-\frac{c}{2}}^{\frac{c}{2}}
  \int_{-\frac{c}{2}}^{\frac{c}{2}} f'(x) \ln |x-y|^{-1} \delta f_{n}'(y)
  \dx\,\dy.
\end{equation}
We need to estimate the terms $Q[\delta f_{n}]$ and $L[\delta f_{n}]$.
First consider $L[\delta f_{n}]$. Note, that both $f(x)$ and
$f_{n}(x)$ satisfy the condition $\int_{-c/2}^{c/2}f'(x)\dx=c-4$
\eqref{eq:74}. Thus for $\delta f_{n}(x)$ we have
\begin{equation}
    \label{eq:91}
    \int_{-c/2}^{c/2}\delta f_{n}'(x)\dx=0.
\end{equation}
Also for $|y|<\sqrt{2c-4}$ we can substitute $f'(x)=1-4\rho(x)$ and due to
Proposition \ref{theorem:deift-min-suff-cond} and equations
\eqref{eq:23}, \eqref{eq:72} we see that the integral over $x$ 
is a constant $c-4\ell$:
  \begin{equation}
    \label{eq:92}
    \int_{-c/2}^{c/2}f'(x)\ln|x-y|^{-1}\dx = c-4\ell.
  \end{equation}
If we subtract $\frac{1}{16}\int_{-c/2}^{c/2}(c-4\ell)\delta f_{n}'(y)\dy=0$ from $L[\delta f_{n}]$, we get
  \begin{equation}
    \label{eq:93}
L[\delta f_{n}]=\frac{1}{16} \int_{-\frac{c}{2}}^{\frac{c}{2}}\left[
  \int_{-\frac{c}{2}}^{\frac{c}{2}} f'(x) \ln |x-y|^{-1} 
  \dx-(c-4\ell)\right]\,\delta f_{n}'(y)\dy.
  \end{equation}
The expression inside the brackets is zero if $|y|<\sqrt{2c-4}$ and less than
zero if $y>\sqrt{2c-4}$ (greater than zero for $y<-\sqrt{2c-4}$). Since
$f'(x)=1$ and $f_{n}'(x)=1$ for $|x|>\sqrt{2c-4}$, we see that
$\delta f_{n}'(y)\leq 0$ for $y>\sqrt{2c-4}$ (and greater than zero for
$y<-\sqrt{2c-4}$) and conclude that value of $L[\delta f_{n}]$   is
non-negative.

Since $\mathrm{supp}\,\delta f_{n}\subset
[-\sqrt{2c-4},\sqrt{2c-4}]$, i.e. $\mathrm{supp}f_{n}\subset
\mathrm{supp} f$, the term $L[\delta f_{n}]$ is exactly
zero.

Now we need to estimate the contribution $Q[\delta f_{n}]$. To do so
we need the inequality
\begin{equation}
  \label{eq:64}
  -\int_{-\frac{c}{2}}^{\frac{c}{2}}\int_{-\frac{c}{2}}^{\frac{c}{2}}
  \delta f_{n}'(x)\delta f_{n}'(y)\ln |x-y|
  \dx\;\dy\leq \frac{C}{2}\int_{-\infty}^{\infty}\int_{-\infty}^{\infty}\left(\frac{\delta
    f_{n}(x)-\delta f_{n}(y)}{x-y}\right)^{2} \dx\;\dy,
\end{equation}
for some constant $C$, that is valid for all piecewise continuously
differentiable functions with a compact support.
This inequality can be obtained using the Proposition 1.15 in
\cite{romik2015surprising}, where it is shown that
$Q[f]=\frac{1}{64}\int_{-\infty}^{\infty}|x|\cdot
|\hat{f}(x)|^{2}\dx$, where $\hat{f}$ is a Fourier transform of $f$.
Then the left-hand side of the inequality is a definition of a norm on a
Sobolev space $H^{\frac{1}{2}}$. Right-hand side is the Slobodeckij
norm on the same space and it is an equivalent one
\cite{ladyzhenskaia1968linear,lions2012non,heuer2014equivalence}, so
the inequality follows. It is also possible to prove that it is
actually the equality with $C=1$, using the integration by parts (See e.g.\cite{bufetov2010diagrams} Lemma 4).

Since $|\delta f_{n}(x)|\leq \frac{\sqrt{2}}{n}$ for all $x$, we can
estimate the right hand side of the equation \eqref{eq:64} by dividing
the domain of integration into two parts: ${|x-y|>\frac{\sqrt{2}}{n}}$,
where we us this estimate and obtain contribution of the order
$\frac{C}{n}$ and $|x-y|\leq \frac{\sqrt{2}}{n}$, where we use
Lipschitz property of $\delta f_{n}$ and  obtain the
contribution of the order $\frac{\tilde{C}\ln n}{n}$. Thus we see that
$  Q[\delta f_{n}]=\mathcal{O}\left(\frac{\ln n}{n}\right)$ and we
conclude that $\left\{-\frac{1}{(2n)^{2}}\ln
  \mu_{n}(\lambda_{n})\right\}$ converges to $J[f]$, so $J[f]\geq 0$.
\end{proof}
\begin{lemma}
  \label{lemma:eps-probability}
  For a highest weight $\lambda$ with the boundary of rotated Young diagram
  given by a function $f_{n}(x)$ such that $d(f_{n},f)=\varepsilon$, the
  probability is bounded by
  \begin{equation}
    \label{eq:88}
    \mu_{n}(\lambda)\leq C_{2} e^{-n^{2}\varepsilon^{2}+\mathcal{O}(n\ln n)}.
  \end{equation}
\end{lemma}
\begin{proof}
Denote by $\delta f_{n}$ a difference of $f_{n}$ and $f$: $\delta f_{n}(x)=f_{n}(x)-f(x)$.
The function $\delta f_{n}(x)$ is a Lipschitz with a compact support.
Then from Lemma \ref{lemma:func-form-of-prob-1} we can write the probability
$\mu_{n}(\lambda)=\exp\left(-(2n)^{2}J[f_{n}]+\mathcal{O}(n\ln
  n)\right)$ in the same way as in Lemma \ref{lemma:functional-greater-than-zero}
as
  \begin{equation}
   \label{eq:66}
    \mu_{n}(\delta f_{n})=e^{\left(-(2n)^{2} \left(
        J[f]+Q[\delta f_{n}] +L[\delta f_{n}]\right)\right)}\cdot\exp\left(\mathcal{O}(n\ln n)\right).
  \end{equation}
By the condition of the lemma we have $Q[\delta f_{n}]=\varepsilon^{2}$. 

Similarly to Lemma \ref{lemma:functional-greater-than-zero} we
demonstrate that the term
$$L[\delta f_{n}]:=\frac{1}{16}
\int_{-\frac{c}{2}}^{\frac{c}{2}}\int_{-\frac{c}{2}}^{\frac{c}{2}}
f'(x) \ln |x-y|^{-1} \delta f_{n}'(y) \dx\,\dy$$
can be written as
$\frac{1}{16} \int_{-\frac{c}{2}}^{\frac{c}{2}}\left[
  \int_{-\frac{c}{2}}^{\frac{c}{2}} f'(x) \ln |x-y|^{-1}
  \dx-(c-4\ell)\right]\,\delta f_{n}'(y)\dy.$ The expression inside the
brackets is zero if $|y|<\sqrt{2c-4}$ and less than zero if
$y>\sqrt{2c-4}$ (greater than zero for $y<-\sqrt{2c-4}$), since
logarithm is a monotonic function. Since $f'(x)=1$ and
$|f_{n}'(x)|=1$ for $|x|>\sqrt{2c-4}$, we see that
$\delta f_{n}(y)\leq 0$ for $y>\sqrt{2c-4}$ (and greater than zero for
$y<-\sqrt{2c-4}$) and conclude that value of $L[\delta f_{n}]$ is
non-negative. Also if
$\mathrm{supp}\,\delta f_{n}\subset [-\sqrt{2c-4},\sqrt{2c-4}]$ this
contribution is exactly zero.

From Lemma
\ref{lemma:functional-greater-than-zero} we have $J[f]\geq 0$,
so we see that ${\exp\left(-(2n)^{2}\left(J[f]+L[\delta
    f_{n}]\right)\right)<C_{2}}$ and the lemma is proven.
\end{proof}

\begin{lemma}
  As $n\to\infty$ rotated Young diagrams for highest weights in the decomposition
  of tensor power of the spinor representation of simple Lie algebra $\son$
  into irreducible representations converge in probability in the metric
  $d_{Q}$ to the limiting shape given by the formulas \eqref{eq:62},
  \eqref{eq:61}, \eqref{eq:68}. That is, for all $\varepsilon>0$ we have
  \begin{equation}
    \label{eq:67}
    \mathbb{P}\left(||f_{n}-f||_{Q}>\varepsilon\right)\xrightarrow[n\to\infty]{} 0.
  \end{equation}

\end{lemma}
\begin{proof}
  By Lemma \ref{lemma:eps-probability} the probability of each highest weight
  $\lambda$ with a rotated Young diagram with boundary $f_{n}$ such that
  $||f_{n}-f||_{Q}>\varepsilon$ is bounded by
  $e^{-n^{2}\varepsilon^{2}+\mathcal{O}(n\ln n)}$. We also know that number of
  such highest weights is not greater than the total number of dominant
  integral weights in the reducible representation
  $\left(V^{\omega_{n}}\right)^{\otimes N}$. Let us estimate this number.
  Weight diagram of $\left(V^{\omega_{n}}\right)^{\otimes N}$ is an
  $n$-dimensional hypercube with a side $2N$ which contains at most $(2N)^{n}$
  integral weights. But this hypercube is then divided into Weyl chambers.
  Total number of Weyl chambers is equal to the order of Weyl group of
  $B_{n}$, which is $2^{n}n!$. Thus we can estimate the number of integral
  dominant weights as being not greater than
  $\frac{2^{n}N^{n}}{2^{n} n!}<e^{n\ln n}$. Combining the bound on number of
  highest weights and the bound on probability we come to the conclusion that
  \begin{equation}
    \label{eq:95}
    \mathbb{P}\left(||f_{n}-f||_{Q}>\varepsilon\right)<e^{-n^{2}\varepsilon^{2}+\mathcal{O}(n\ln
    n)}\xrightarrow[n\to\infty]{} 0.
  \end{equation}

\end{proof}

Now from this lemma and from \eqref{eq:87} follows Theorem \ref{diagram-convergence-theorem}.

We have obtained the limit shape for Young diagrams for tensor product
decomposition of tensor powers of last fundamental representation of Lie
algebras of series $\son$. In the Figure \ref{fig:Bn-limit-shapes-3} we
present limit shapes for $c=3$, $c=4$ and $c=6$ as well as the most probable
diagram for $n=20$ and $N=40$ ($c\approx 4$).

\begin{figure}
  \label{fig:Bn-limit-shapes-3}
   \includegraphics{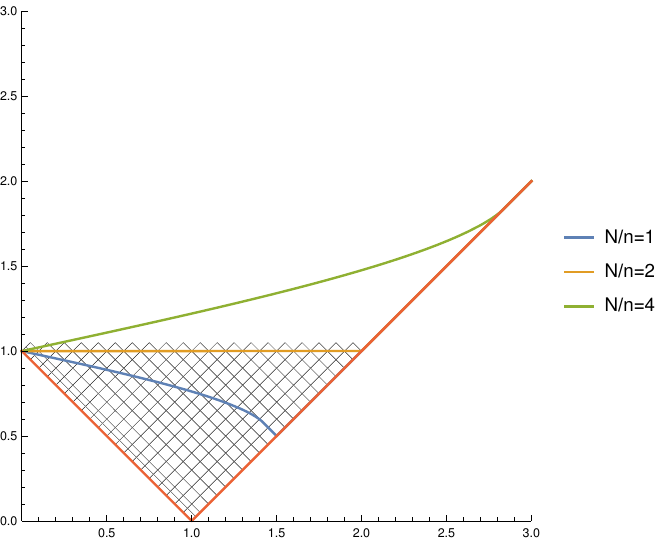}
   \caption{Limit shapes for Young diagrams for $c=3$ (Blue), $c=4$ (Yellow)
     and $c=6$ and a most probable diagram for $n=20, N=40$. Plots
     are by the explicit formulas (\ref{eq:68},\ref{eq:62})}
\end{figure}

\section{Proof of the Central limit theorem}
\label{sec:centr-limit-theor}


In general, Central Limit Theorems for the determinantal point
processes can be proven by the use of discrete orthogonal polynomials
and deducing their continuous asymptotic, as described in the
book~\cite{baik2007discrete} and many papers, for
example~\cite{breuer2017central}. In our case the ensemble is a little
bit different, since the measure \eqref{meas} contains a Vandermonde
determinant written in terms of squares of variable $a_{i}$. Therefore
some care is required, below we describe how to adapt the results
of~\cite{breuer2017central} to our case.

Introduce the variables $y_{i}=\frac{a_{i}^{2}}{(2n)^{2}}$, the
measure \eqref{meas} takes the determinantal form
\begin{equation}
  \label{eq:109}
  \mu_{n,N}(y_{1},\dots,y_{n})=\frac{1}{Z_{N,n}}\prod_{i<j}(y_{i}-y_{j})^{2}\cdot
  \prod_{l=1}^{n} y_{l}
  \frac{(N+2n-1)!}{\left(\frac{N+2n-1}{2}-\frac{2n\sqrt{y_{l}}}{2}\right)!
    \left(\frac{N+2n-1}{2}+\frac{2n\sqrt{y_{l}}}{2}\right)!},
\end{equation}
where we have rewritten the weight in the binomial form and collected
the terms that do not depend on $y_{i}$ into the normalization
constant $Z_{N,n}$. Denote the weight by $W^{(n)}(y)$:
\begin{equation}
  \label{eq:121}
  W^{(n)}(y)=y
  \frac{(N+2n-1)!}{\left(\frac{N+2n-1}{2}-\frac{2n\sqrt{y}}{2}\right)!
  \left(\frac{N+2n-1}{2}+\frac{2n\sqrt{y}}{2}\right)!},
\end{equation}
and introduce the discrete orthonormal polynomials
$p^{(n)}_{m}(y)$ on a quadratic lattice:
\begin{equation}
  \label{eq:119}
  \sum_{i=1}^{N+2n-1} p^{(n)}_{m}\left(\frac{i^{2}}{(2n)^{2}}\right)
  p^{(n)}_{l}\left(\frac{i^{2}}{(2n)^{2}}\right)
  W^{(n)}\left(\frac{i^{2}}{(2n)^{2}}\right) = \delta_{m,l}. 
\end{equation}
The orthogonal polynomials satisfy the three-term recurrence relation
\begin{equation}
  \label{eq:122}
  xp^{(n)}_{m}(x)=a^{(n)}_{m+1}p^{(n)}_{m+1}(x)+b^{(n)}_{m}p^{(n)}_{m}(x)+a^{(n)}_{m}p^{(n)}_{m-1}(x). 
\end{equation}
The measure is then rewritten as a determinantal ensemble $\mathcal{P}^{(n)}$:
\begin{equation}
  \label{eq:120}
    \mu_{n,N}(y_{1},\dots,y_{n})=\det\left(\sqrt{W^{(n)}(y_{i}) W^{(n)}(y_{j})}
      \sum_{l=0}^{n-1}p^{(n)}_{l}(y_{i})\;
      p^{(n)}_{l}(y_{j})\right)_{i,j=1}^{n}.
\end{equation}

To prove the Central limit theorem we would like to apply the following theorem:
\begin{theorem}[Theorem 2.5 in \cite{breuer2017central}]
  \label{thm:breuer-clt}
  Let $\{p^{(n)}_{m}(x)\}_{m=0}^{n-1}$ be normalized orthogonal
  polynomials of the polynomial ensemble $\mathcal{P}_{n}$, that
  satisfy the three-term recurrence relation
  \begin{equation}
    \label{eq:125}
    xp^{(n)}_{m}(x)=a^{(n)}_{m+1}p^{(n)}_{m+1}(x)+b^{(n)}_{m+1}p^{(n)}_{m}(x)+a^{(n)}_{m}p^{(n)}_{m-1}(x),
  \end{equation}
  and assume that there exists a subsequence $\{n_{j}\}_{j}$ and
  $a>0, b\in\mathbb{R}$ such that for any $k\in\mathbb{Z}$ we have
  \begin{equation}
    \label{eq:126}
    a^{(n_{j})}_{n_{j}+k}\to a, \quad b^{(n_{j})}_{n_{j}+k}\to b, 
  \end{equation}
  as $j\to\infty$. Then for any real-valued $f\in C^{1}(\mathbb{R})$
  we have
  \begin{equation}
    \label{eq:127}
    X_{f}^{(n_{j})}-\mathbb{E}X_{f}^{(n_{j})}\to
    \mathcal{N}\left(0,\sum_{l\geq 1}l|\hat{f}_{l}|^{2}\right), \quad
    \text{as}\; j\to\infty,
  \end{equation}
  in distribution, where the coefficients $\hat{f}_{l}$ are defined as
  \begin{equation}
    \label{eq:128}
    \hat{f}_{l}=\frac{1}{2\pi}\int_{0}^{2\pi} f(2a\cos \theta +b) e^{-il\theta}d\theta,
  \end{equation}
  for $l\geq 1$. When $n_{j}=j$, that is the subsequence is the whole
  sequence, \eqref{eq:126} is equivalent to
  \begin{equation}
    \label{eq:129}
    a^{(n)}_{n}\to a,\quad b^{(n)}_{n}\to b.
  \end{equation}
\end{theorem}

To use this theorem we need to establish the convergence of the
recursion coefficients. We do this by relating the polynomials
$p^{(n)}_{m}$ to the Krawtchouk polynomials using the ``lifting''
procedure from \cite{buhmann1992orthogonal} which is a variant of the
QR-algorithm for the Christoffel transformation of orthognal
polynomials~\cite{galant1971implementation},~\cite{bueno2004darboux}.

First note, that the polynomials $p^{(n)}_{m}(y)$ with $y=x^{2}$ are
just even polynomials from the set of polynomials
$\widetilde{p}^{(n)}_{k}(x)$ that are orthonormal with respect to the
weight $W^{(n)}(x^{2})$:
\begin{equation}
  \label{eq:124}
  \sum_{i=-(N+2n-1)}^{N+2n-1} \widetilde{p}^{(n)}_{m}\left(\frac{i}{2n}\right)
  \widetilde{p}^{(n)}_{l}\left(\frac{i}{2n}\right)
  W^{(n)}\left(\frac{i^{2}}{(2n)^{2}}\right) = \delta_{m,l}.   
\end{equation}
Since weight $W^{(n)}(x^{2})$ is even as a function of $x$ and since
each value of $y$ appears in \eqref{eq:124} twice, we have

\begin{equation}
  \label{eq:123}
  p^{(n)}_{m}\left(x^{2}\right)=\sqrt{2}\widetilde{p}^{(n)}_{2m}(x). 
\end{equation}
Then the recursion coefficients $\widetilde{b}^{(n)}_{k}=0$ for all $k$ for
the polynomials $\widetilde{p}^{(n)}_{k}$, and the recursion
coefficients $a^{(n)}_{m}, b^{(n)}_{m}$ are expressed in terms of
$\widetilde{a}^{(n)}_{k}$ as
\begin{equation}
  \label{eq:130}
  \begin{array}{l}
    a^{(n)}_{m}=\widetilde{a}^{(n)}_{2m}\;\widetilde{a}^{(n)}_{2m-1}\\
    b^{(n)}_{m}=\left(\widetilde{a}^{(n)}_{2m}\right)^{2}+\left(\widetilde{a}^{(n)}_{2m+1}\right)^{2}.
  \end{array}
\end{equation}
From now on it is more convenient to work with monic orthogonal
polynomials $\widetilde{P}^{(n)}_{m}(x)$, that correspond to the
orthonormal polynomials $\widetilde{p}^{(n)}_{m}(x)$ and satisfy the
recurrence relation
\begin{equation}
  \label{eq:131}
  x \widetilde{P}^{(n)}_{m}(x)=\widetilde{P}^{(n)}_{m+1}(x)+
  \left(\widetilde{a}^{(n)}_{m}\right)^{2}\widetilde{P}^{(n)}_{m-1}(x).
\end{equation}

Now consider monic Krawtchouk polynomials $\widetilde{K}_{l}(z;q,M)$
that satisfy the orthogonality relation
\cite{koekoek2010hypergeometric}
\begin{equation}
  \label{eq:132}
  \sum_{i=0}^{M}\binom{M}{i} q^{i}(1-q)^{M-i}\widetilde{K}_{l}(i;q,M)
  \widetilde{K}_{j}(i;q,M)= \delta_{l,j}\cdot q^{l}(1-q)^{l}\; l!\;\cdot\prod_{i=0}^{l-1}(M-i).
\end{equation}
and recurrence relation
\begin{multline}
  \label{eq:133}
  z\widetilde{K}_{m}(z;q,M)=\widetilde{K}_{m+1}(z;q,M)+
  \left[q(M-m)+m(1-q)\right]\widetilde{K}_{m}(z;q,M)+\\
  +mq(1-q)(M+1-m)\widetilde{K}_{m-1}(z;q,M),
\end{multline}
for $z=0,\dots,M$. Take $q=\frac{1}{2}$, $M=N+2n-1$, shift and rescale the
argument $\widetilde{z}=\frac{1}{n}(z-\frac{N+2n-1}{2})$, then monic
orthogonal polynomials
$\widetilde{K}^{*}_{m}(\widetilde{z})=
n^{-m}\widetilde{K}_{m}\left(\frac{N+2n-1+2n\widetilde{z}}{2};\frac{1}{2},N+2n-1\right)$
have weight
$W^{*}(\widetilde{z})=\frac{(N+2n-1)!}{\left(\frac{N+2n-1-2n\widetilde{z}}{2}\right)!
  \left(\frac{N+2n-1+2n\widetilde{z}}{2}\right)!}$ and satisfy the recursion relation
\begin{equation}
  \label{eq:134}
  \widetilde{z}\widetilde{K}^{*}_{m}(\widetilde{z})=\widetilde{K}^{*}_{m+1}(\widetilde{z})
  +\frac{m}{4n^{2}}(N+2n-m)\widetilde{K}^{*}_{m-1}(\widetilde{z}).
\end{equation}
We see that the weight $W^{(n)}(x^{2})$ of the polynomials
$\widetilde{P}^{(n)}_{m}$ differs from the weight $W^{*}(x)$ of
$\widetilde{K}^{*}_{m}$ by a factor $x^{2}$. Therefore the polynomials
$\widetilde{P}_{m}$ are expressed in terms of the polynomials
$\widetilde{K}^{*}_{m}$ and their derivatives $\left(\widetilde{K}^{*}_{m}\right)'$ by the
Christoffel transformation \cite[section 2.5]{ismail2020encyclopedia}:
\begin{equation}
  \label{eq:135}
  \widetilde{P}^{(n)}_{m}(x)=\frac{1}{C_{m,2} x^{2}}\left|
  \begin{matrix}
    \widetilde{K}^{*}_{m}(0)&     \widetilde{K}^{*}_{m+1}(0) &
    \widetilde{K}^{*}_{m+2}(0)\\
    \left(\widetilde{K}^{*}_{m}\right)'(0)&     \left(\widetilde{K}^{*}_{m+1}\right)'(0) &
    \left(\widetilde{K}^{*}_{m+2}\right)'(0)\\
    \widetilde{K}^{*}_{m}(x)&     \widetilde{K}^{*}_{m+1}(x) &
    \widetilde{K}^{*}_{m+2}(x)
  \end{matrix}
  \right|,
\end{equation}
where
\begin{equation}
  \label{eq:136}
  C_{m,2}=\left|
    \begin{matrix}
      \widetilde{K}^{*}_{m}(0)&     \widetilde{K}^{*}_{m+1}(0) \\
    \left(\widetilde{K}^{*}_{m}\right)'(0)&     \left(\widetilde{K}^{*}_{m+1}\right)'(0) 
  \end{matrix}
  \right|.
\end{equation}
Such transformation of polynomials leads to the following
transformation of recursion coefficients: take three-diagonal matrix
of the recursion coefficients, compute its QR-factorization and
multiply the factors in reversed order
\cite{galant1971implementation}.

In the case of monic polynomials with recurrent relation of the form
$
\widetilde{z}\widetilde{K}^{*}_{m}(\widetilde{z})=\widetilde{K}^{*}_{m+1}(\widetilde{z})
+\beta_{m}\widetilde{K}^{*}_{m-1}(\widetilde{z})$ this transformation
is given by the relations \cite{buhmann1992orthogonal}[formula (2.10)]:  
\begin{equation}
  \label{eq:137}
  \begin{array}{l}
    \widetilde{\beta}_{2m}=\frac{\beta_{2m}\beta_{2m+1}}{\widetilde{\beta}_{2m-1}},\\
    \widetilde{\beta}_{2m+1}=\beta_{2m+1}+\beta_{2m+2}-\widetilde{\beta}_{2m},
  \end{array}
\end{equation}
where $\beta_{m}=\frac{m}{4n^{2}}(N+2n-m)$.
We were not able to solve these relations explicitly since the
coefficients $\widetilde{\beta}_{m}$ are given by the continuous fractions
of $\beta_{l}$ and are cumbersome, but we can write asymptotic
solutions in the form
\begin{equation}
  \label{eq:138}
  \begin{array}{l}
    \widetilde{\beta}_{2m}=\beta_{2m}-\frac{m}{3n^{2}}+\mathcal{O}\left(\frac{1}{n^{2}}\right),\\
    \widetilde{\beta}_{2m+1}=\beta_{2m+1}+\frac{N}{2n^{2}}-\frac{5m}{3n^{2}}-\frac{1}{2n^{2}}+
    \mathcal{O}\left(\frac{1}{n^{3}}\right).
  \end{array}
\end{equation}
Substituting these asymptotics to the equation \eqref{eq:137} for
$\widetilde{\beta}_{2m+2}$ we check that
\begin{multline}
  \label{eq:139}
  \widetilde{\beta}_{2m+2}=\frac{\beta_{2m+2}\beta_{2m+3}}{\beta_{2m+1}+
    \frac{N}{2n^{2}}-\frac{5m}{3n^{2}}-\frac{1}{2n^{2}}+
    \mathcal{O}\left(\frac{1}{n^{3}}\right)}= \\
  =\frac{\beta_{2m+2}\beta_{2m+3}}{\beta_{2m+3}+\frac{m+1}{3n^{2}}+\frac{2}{3n^{2}}
    +\mathcal{O}\left(\frac{1}{n^{2}}\right)}=
  \\
  =\beta_{2m+2}+\frac{m+1}{3n^{2}}
  +\mathcal{O}\left(\frac{1}{n^{2}}\right).     
\end{multline}
Therefore for $m=2n$ we have the asymptotics
$\widetilde{\beta}_{2n}=\frac{N}{2n}+\mathcal{O}\left(\frac{1}{n}\right)=\frac{c-2}{2}+\mathcal{O}\left(\frac{1}{n}\right)$
and thus using \eqref{eq:130},\eqref{eq:131} we conclude that $a^{(n)}_{n}\to \frac{c-2}{2}$ and
$b^{(n)}_{n}\to c-2$. Then we can apply Theorem \ref{thm:breuer-clt}
to conclude the proof of the Central limit theorem. 

\section*{Conclusion}

We have proven the convergence of irreducible components in tensor
powers of the spinor representation of $\son$ to the limit shape. It
is possible to write an alternative proof based on the general results
on discrete beta-ensembles, presented in the book by Alice Guionnet
\cite{guionnet2019asymptotics}. 

Similar limit shapes can be obtained for the tensor powers of the
certain reducible representations of $\gln,\sond,\spn$. This result
is presented in the paper by A.~Nazarov, O.~Postnova
and T.~Scrimshaw \cite{2021arXiv211112426N}.

It would be interesting to obtain similar result for the tensor powers
of other irreducible representations and for Lie algebras of other
classical series. Unfortunately, there are no known explicit formulas
for the tensor product decomposition coefficients for the tensor
powers of the irreducible representations in the cases except
$V^{\omega_{1}}$ for $A_{n-1}$ and $V^{\omega_{n}}$ for $\son$. So
this generalization remains an unsolved problem for the future.

Another possible generalization is to consider the character measure
\begin{equation}
  \label{eq:104}
   p^N(\lambda,t)=\frac{M^{N}_{\lambda}\chi_{\lambda}(e^{t})}{\chi_{\nu}(e^{t})^{N}}. 
\end{equation}
In the paper \cite{tate2004lattice} an asymptotic formula for tensor
product decomposition coefficients was obtained. The character measure
was studied in the case of Lie albegras of fixed rank $n$ in the
papers \cite{Postnova2019}. In a separate publication we will consider
the limit $n,N\to\infty$ for the character measure.

It would be also very interesting to establish the connection of the present limit shape
with random matrices. Such a connection is well known for Vershik-Kerov-Logan-Shepp
limit shape \cite{borodin2001robinson}, \cite{bufetov2013kerov}.

We also plan to study the local fluctuations around the limit shape
presented in this paper. We expect to obtain an analogue of
Baik-Deift-Johansson theorem \cite{baik1999distribution}.

The entropy of the Plancherel measure for the $S_{n}$-representations
and $\sln{n}$-representations was considered in the papers
\cite{vershik1985asymptotic}, \cite{bufetov2012vershik},
\cite{mkrtchyan2014entropy}. Our preliminary numerical calculations
demonstrate that similar result holds for the series $\son$. 

\section*{Acknowledgements}

We are grateful to professor Nikolai Reshetikhin and professor Anatoly
Vershik for the attention to this work. We thank Pavel Belov and
Sergey Simonov for useful discussions.

The authors thank the Sirius Mathematics Center for their hospitality
and support during our stay at Sirius where some work on this paper
has been done. We acknowledge hospitality and support from Galileo
Galilei Institute, and from the scientific program on "Randomness,
Integrability, and Universality", where some of the results were
presented.

Pavel Nikitin is supported by the Russian Science Foundation under
grant No.~21-11-00152.

\printbibliography

\end{document}